\RequirePackage{fix-cm}
\documentclass[smallextended]{svjour3}       
\smartqed  
\usepackage{amsmath}
\usepackage{amsbsy}
\usepackage{amssymb}
\usepackage{textcomp}
\usepackage{graphicx}
\usepackage{amsfonts}
\usepackage{longtable}
\usepackage{multirow}
\usepackage{lscape}
\usepackage{url}
\usepackage{citeref}
\usepackage{subfig}
\usepackage{longtable}
\usepackage{lscape}
\usepackage{hhline}
\usepackage[utf8]{inputenc}
\usepackage{colortbl,array} 
\usepackage{multirow,bigdelim}
\usepackage{booktabs}
\usepackage[linesnumbered,ruled]{algorithm2e}
\usepackage{epstopdf}
\newtheorem{thm}{Theorem.\nopagebreak}
\newtheorem{prop}[thm]{Proposition\nopagebreak}

\newtheorem{cor}[thm]{Corollary.\nopagebreak}

\def\eeq{\end{equation}} 
\def\lbeq#1{\begin{equation}\label{#1}} 

\def\eps{\varepsilon} 
\def\bealg#1{\begin{alg}{\bf #1}\index{Algorithm!#1}\nopagebreak}
\def\ealg{\end{alg}}
\def\D{\displaystyle}				

\def\fct#1{\mathop{\rm #1}}	

\def\argmin{\fct{argmin}}

\def\wh{\widehat}
\def\wt{\widetilde}
\def\bary{\begin{array}}
\def\eary{\end{array}}

\def\st{\fct{s.t.~}}
\def\x{\mathbf{x}}

\def\<{\langle}
\def\>{\rangle}

\begin{document}

\title{Accelerated first-order methods for large-scale convex minimization 
}
\author{Masoud Ahookhosh}


\institute{Faculty of Mathematics, University of Vienna,
Oskar-Morgenstern-Platz 1, 1090 Vienna, Austria\\
\email{masoud.ahookhosh@univie.ac.at}
}


\maketitle
\vspace{-4mm}
\begin{abstract}
This paper discusses several (sub)gradient methods attaining the optimal complexity for smooth problems with Lipschitz continuous gradients, nonsmooth problems with bounded variation of subgradients, weakly smooth problems with
H\"older continuous gradients. The proposed schemes are optimal for smooth strongly convex problems with Lipschitz continuous gradients and optimal up
to a logarithmic factor for nonsmooth problems with bounded variation of subgradients. More specifically, we propose two estimation sequences of the
objective and give two iterative schemes for each of them. In both cases, the
first scheme requires the smoothness parameter and the H\"older constant, while
the second scheme is parameter-free (except for the strong convexity parameter
which we set zero if it is not available) at the price of applying a
nonmonotone backtracking line search. A complexity analysis for all the
proposed schemes is given. Numerical results for some applications in sparse
optimization and machine learning are reported, which confirm the theoretical
foundations.

\end{abstract}

\section{Introduction} \label{e.intro}
Let $V$ be a finite-dimensional linear vector space with the dual space $V^*$ as the space of all linear function on
$V$. We assume 
$f:V\rightarrow\overline{\mathbb{R}}:=\mathbb{R}\cup\{+\infty\}$ 
is a  proper, $\mu_f$-strongly convex ($\mu_f>0$ for strongly convex case and $\mu_f=0$ for convex case), and lower semicontinuous function satisfying
\lbeq{e.holder}
\|\nabla f(x)-\nabla f(y)\|_* \leq L_{\nu} \|x-y\|^{\nu} ~~~\forall~ x,y 
\in V,
\eeq
where $\nabla f(x)$ denotes the gradient of $f$ at $x$ for $\nu \in (0,1]$ or
any subgradient of $f$ at $x$ ($\nabla f(x) \in \partial f(x)$) for $v=0$. Let the function $\psi: V \rightarrow \overline{\mathbb{R}}$ be simple, proper, $\mu_p$-strongly convex ($\mu_p\geq 0$), and lower semicontinuous function. We consider the structured convex minimization problem 
\lbeq{e.gfun}
\begin{array}{ll}
\min          &~ h(x):=f(x)+\psi(x)\\
\mathrm{s.t.} &~ x \in C,
\end{array}
\eeq
where $C$ is a simple, nonempty, closed, and convex set. By (\ref{e.holder}),
we have  $f\in \mathcal{C}_{\mu_f, L_\nu}^{1,\nu}(V)$, i.e., $f$ can be smooth
with Lipschitz continuous gradients ($\nu=1$), weakly smooth problems with
H\"older continuous gradients ($\nu \in {]0,1[}$), or nonsmooth with bounded
variation of subgradients ($\nu=0$). Hence the objective $h$ is $\mu$-strongly convex
with $\mu:=\mu_f+\mu_p\geq 0$. We assume that the first-order black-box oracle
of the objective $h$ is available.

\subsection{{\bf Motivation \& history}} \label{e.moth}
Over the past few decades, due to the dramatic increase in the size of data for many applications, first-order methods have been received much attention thanks to their simple structures and low memory requirements. The efficiency of first-order methods can be poor (a large number of function values and subgradients is needed) for solving the general convex problems if the structure of the
problem is not available. As a result, to develop practically appealing
schemes, it is necessary to make an additional restriction on problem classes.
In particular, developing efficient methods for
solving large-scale convex optimization problems is possible if the underlying objective has a suitable structure and the domain is simple enough. 
Convexity and level of smoothness are two important factors playing key roles in construction of efficient schemes for such structured optimization problems. 

Let $x^*$ be an optimizer of (\ref{e.gfun}) and $x_k$ be an approximate
solution given by a first-order method. We call $x_k$ an $\eps$-solution of
(\ref{e.gfun}) if $h(x_k)-h(x^*)\leq \eps$, for a prescribed accuracy 
parameter $\eps>0$. In 1983, {\sc Nemirovski \& Yudin} in \cite{NemY} derived
optimal worst-case complexities for first-order methods to achieve an
$\eps$-solution for several classes of convex problems (see Table  \ref{t.complexity}). If a first-order scheme attains the worst-case complexity
of a class of problems, it is called optimal. A special feature of these
methods is that the corresponding complexity does not depend explicitly on the
problem dimension. From practical point of view, studying the effect of an uniform boundedness of the complexity is very attractive and such methods are
highly recommended when the prescribed accuracy $\eps$ is not too small,
whereas the dimension of problem is considerably large.

\vspace{-4mm}
\begin{table}[htbp]
\caption[Convex optimization problems and their complexity]
{List of the best known complexities of first-order methods for several classes
of problems with respect to levels of smoothness and convexity (cf. \cite{NemN,NemY,NesB})}
\vspace{-2mm}
\label{t.complexity}
\begin{center}\footnotesize
\renewcommand{\arraystretch}{1.3}
\begin{tabular}{|l|c|c|}\hline
\multicolumn{1}{|l|}{{\bf Problem's class}} 
& \multicolumn{1}{c|}{{\bf Convex problems}}
& \multicolumn{1}{c|}{{\bf Strongly convex problems}}\\ 
\hline
Smooth problems ($\nu=1$) 
& $\mathcal{O}(\eps^{-1/2})$ & $\mathcal{O}(\ln (1/\eps))$\\
\hline
Weakly smooth problems ($\nu \in {]0,1[}$)
& $\mathcal{O}(\eps^{-2/(1+3\nu)})$ & 
$\leq \mathcal{O}(\eps^{-{(1-\nu)}/(1+3\nu)}) \ln \left(1/\eps^{{(3+\nu)}/(1+3\nu)}\right)$\\
\hline
Nonsmooth problems  ($\nu=0$)           
& $\mathcal{O}(\eps^{-2})$ & $\mathcal{O}(1/\eps)$\\
\hline
\end{tabular}
\end{center}
\end{table}

\vspace{-5mm}
In \cite{NemY}, it was proved that subgradient, subgradient projection, and
mirror descent methods possess the optimal complexity $\mathcal{O}(\eps^{-2})$ for Lipschitz continuous nonsmooth problems, where the 
mirror decent method is a generalization of the subgradient projection 
method, cf. \cite{BecT1}. In 1983, the pioneering optimal method by
{\sc Nesterov} \cite{Nes83} was introduced for smooth problems
with Lipschitz continuous gradients. He later in \cite{NesB} proposed
some more gradient methods for this class of problems. 
{\sc Nesterov} in \cite{NesC} proposed a gradient-type method for 
minimizing the composite problem (\ref{e.gfun}) with the 
complexity  $\mathcal{O}(\varepsilon^{-1/2})$, where $f$ has 
Lipschitz continuous gradients and $\psi$ is a simple convex function.
Since 1983 many researchers have developed the idea of optimal schemes, see,
e.g., {\sc Auslander \& Teboulle} \cite{AusT}, {\sc Baes} 
\cite{Bae}, {\sc Baes \& B\"urgisser} \cite{BaeB}, {\sc Beck \& Teboulle} 
\cite{BecT2}, {\sc Chen} et al. \cite{CheLO1,CheLO2}, 
{\sc Gonzaga} et al. \cite{GonK,GonKR}, {\sc Juditsky \& Nesterov} \cite{JudN},
{\sc Lan} \cite{LanS,LanB}, {\sc Lan} et al. \cite{LanLM}, {\sc Nesterov}
\cite{NesB}, {\sc Neumaier} \cite{NeuO} and  {\sc Tseng} \cite{Tse}.
Computational experiments for problems of the form (\ref{e.gfun}) 
have shown that Nesterov-type optimal first-order methods are substantially 
superior to the gradient descent and subgradient methods, cf. {\sc Ahookhosh} \cite{Aho} and {\sc Becker} et al. \cite{BecCG}.
{\sc Nesterov} also in \cite{NesS,NesE} proposed some smoothing 
methods for a class of structured nonsmooth problems attaining the complexity
$\mathcal{O}(\eps^{-1/2})$. 

In 1985, the first optimal method for weakly smooth objectives with H\"older
continuous gradients was given by {\sc Nemirovski \& Nesterov} \cite{NemN}; however, to implement this scheme, one needs to know about $\nu$, $L_{\nu}$, an estimate of the distance of starting point to the optimizer, and the total
number of iterations, which makes the algorithm to some extent impractical. 
{\sc Lan} \cite{LanB} proposed an accelerated bundle-level method attaining the
optimal complexity for all the convex classes considered. This scheme does not
need to know about the global parameters such as Lipschitz or H\"older
constants and the level of smoothness parameter $\nu$; on the other hand, as
the scheme proceeds, the associated auxiliary problem becomes more difficult to solve, i.e., even the limited memory version of this scheme involves solving a computationally costly auxiliary problems. {\sc Devolder} et al. 
\cite{DevGN, DevGN1} also proposed some 
first-order methods for minimization of objectives with H\"older continuous
gradients in inexact oracle. The proposed fast gradient method attains the
optimal complexities for convex problems; however, for implementation it needs to know about $\nu$,
$L_{\nu}$, an estimate of the distance of starting point to the optimizer, and
the total number of iterations. Recently, {\sc Nesterov} \cite{NesU} proposed a
so-called universal gradient method for convex problem classes attaining the
optimal complexities and requiring no global parameters at the price of
applying a backtracking line search. More recently, {\sc Nesterov} proposed a
conditional gradient method involving a simple subproblem, which possesses the
complexity $\mathcal{O}(\eps^{-1/2\nu})$. Moreover,
{\sc Ghadimi} \cite{Gha} and {\sc Ghadimi} et al. \cite{GhaLZ} developed some
first-order methods for unconstrained nonconvex problems of the form
(\ref{e.gfun}) where $f$ is an arbitrary nonconvex function. 

\vspace{-5mm}
\subsection{{\bf Contribution}} \label{s.cont}
This paper describes four accelerated (sub)gradient algorithms (ASGA) attaining optimal
complexities for solving several classes of convex optimization problems with
high-dimensional data (see Table \ref{t.complexity}). 

We firstly construct an estimation sequence using available local or global information of $f$ and then give two iterative schemes for solving (\ref{e.gfun}). The first scheme (ASGA-1) requires the level of smoothness $\nu$ and the H\"older constant $L_{\nu}$. Afterwards, we develop a parameter-free variant of this scheme (ASGA-2) that is not requiring $\nu$ and $L_{\nu}$ at the price of applying a backtracking line search. Apart from an initial point
$x_0$ and the strong convexity parameter $\mu$ ($\mu=0$ if it is not available),
ASGA-2 requires no more parameters. We here emphasize that parameter-free
methods are useful for black-box optimization when no information about $\nu$
and $L_\nu$ is available.

We secondly generalize the estimation sequence of Nesterov \cite{NesU}, by 
adding a quadratic term including the strong convexity information of $h$ and
develop two (sub)gradient methods. The first one
(ASGA-3) needs the smoothness parameters $\nu$ and $L_{\nu}$; on the other
hand, the second one (ASGA-4) is parameter-free at the price of carrying out a
backtracking line search. 

The estimation sequence used in ASGA-1 and ASGA-2 shares some similarities with
the estimation sequence used in ASGA-3 and ASGA-4; however, the iteration sequences used in construction of them are different. Whereas ASGA-1 requires
a single solution of an auxiliary problem, ASGA-3 needs to solve two auxiliary
problems. ASGA-2 requires at least a single solution of an auxiliary problem; in contrast, ASGA-4 needs to solve at least two auxiliary problems. As apposed to ASGA-1 and ASGA-3, the schemes NESUN, ASGA-2, and ASGA-4 are parameter-free (except for strong convexity parameter $\mu$ which we set $\mu=0$ if it is not available) by applying a backtracking line search. NESUN treats strongly convex problems by the same way as convex ones; on the other hand, ASGA-1, ASGA-2,
ASGA-3, and ASGA-4 possess a much better complexity for strongly convex problems.
It is worth mentioning that, for $\mu=0$, ASGA-4 almost reduces to NESUN except
for some parameters.

Apart from some constants, ASGA-1, ASGA-2, ASGA-3, and ASGA-4 possess the same
complexity for finding an $\eps$-solution of the problem (\ref{e.gfun}), i.e.,
\[
\mathcal{O}\left( \mu^{-\frac{1+\nu}{1+3\nu}} L_\nu^{\frac{2}{1+3\nu}}
\eps^{-\frac{1-\nu}{1+3\nu}} \ln(\eps^{-\frac{2}{1+\nu}}) \right),
\]
for $\mu>0$, and 
\[
\mathcal{O}\left(\eps^{-\frac{2}{1+3\nu}} \right),
\]
for $\mu=0$. Therefore, ASGA-1, ASGA-2, ASGA-3, and ASGA-4 are optimal for
smooth, weakly smooth, and nonsmooth convex objectives. On the other hand,
they are optimal for smooth strongly convex problems and optimal
up to a logarithmic factor for nonsmooth strongly convex objectives.
For weakly smooth strongly convex problems, they attain a complexity better
than the known complexity for weakly smooth convex problems.

We finally study the solution of auxiliary problems appearing in ASGA-1,
ASGA-2, ASGA-3, and ASGA-4. The considered auxiliary problems are strongly convex;
however, finding their unique solutions efficiently is highly related to the
structure involved in $\psi$ and $C$. It is shown that these auxiliary problems
can be solved either in a closed form or by a simple iterative scheme for
several functions $\psi$ and domain $C$ appearing in applications. Some
computational experiments show that the performance of ASGA-2, ASGA-4, and
NESUN are sensitive to large regularization parameters and small $\eps>0$ (because of the associated line searches); whereas, ASGA-1
and ASGA-3 are less sensitive. In addition, there are many applications with
available smoothness parameters $\nu$ and $L_\nu$ motivating the quest for
designing ASGA-1 and ASGA-3. It is worth noting that the proposed schemes are able to handle sum of nonsmooth functions, where they behave much better than the traditional subgradient methods in spite of attaining the same complexity
(see Section \ref{s.svm}). Some encouraging numerical results are reported confirming the achieved theoretical foundations.
 
The remainder of the paper is organized as follows. In the next section we 
give two single-subproblem accelerated (sub)gradient schemes with their complexity analysis. In
Section 3 we generalize the estimation sequence of Nesterov \cite{NesU} and
propose two double-subproblem accelerated (sub)gradient schemes and the related complexity
analysis. In Section 4 we discuss the solution of the auxiliary problems
appearing in the proposed methods. In Section 5 we reports some numerical
experiments and comparisons showing the performance of the proposed methods.
Finally, some conclusions are delivered in Section 6.

\subsection{{\bf Preliminaries \& notation}} \label{s.pren}
Let the primal space $V$ be endowed with a norm $\|\cdot\|$, and let the associated dual norm be defined by 
\begin{equation*}
\|s\|_* = \max_{x \in V} \{\langle s,x \rangle ~|~ \|x\| \leq 1\},
\end{equation*}
where $\langle s,x \rangle$ denotes the value of the linear function $s \in V^*$ at $x \in V$. If $V = \mathbb{R}^n$, then, for $1 \leq p \leq \infty$, 
\[
\|x\|_p = \left(\sum_{i=1}^n |x_i|^p\right)^{1/p}.
\]
For a function $\wt f: V \rightarrow \overline{\mathbb{R}} = \mathbb{R}\cup\{\pm \infty\}$, 
$
\mathrm{dom}~ \wt f = \{ x \in V ~|~ \wt f(x) < +\infty\}
$
denotes its effective domain, and $\wt f$ is called proper if $\mathrm{dom}~ \wt f \neq \emptyset$ and $\wt f(x) > -\infty$ for all $x \in V$. Let $C$ be a subset of $V$. In particular, if $C$ is a box, we denote it by $\x = [\underline{x},\overline{x}]$, where in which $\underline{x}$ and $\overline{x}$ are the vectors of lower and upper bounds on the components of x, respectively. The vector $\nabla \wt f(x) \in V^* $ is called a subgradient of $\wt f$ at $x$ if $\wt f(x) \in \mathbb{R}$ and
\[
\wt f(y) \geq \wt f(x) + \langle \nabla \wt f(x),y-x \rangle~~~ \forall y \in V.
\] 
The set of all subgradients is called the subdifferential of $\wt f$ at $x$, 
which is denoted by $\partial \wt f(x)$. 

If $\wt f$ is nonsmooth and convex, then Fermat-type optimality condition for the nonsmooth convex optimization problem
\[
\begin{array}{ll}
\min          &~ \wt f(x)\\
\mathrm{s.t.} &~ x \in C
\end{array}
\]
is given by
\lbeq{e.fopt}
0 \in \partial \wt f(x) + N_C(x),
\eeq
where $N_C(x)$ is the normal cone of $C$ at $x$, i.e.,
\lbeq{e.normc}
N_C(x) := \{p \in V \mid \langle p, x-z \rangle \geq 0 ~~ \forall z \in C\}.
\eeq
For $C\subseteq V$ and $y\in V$, the orthogonal projection is given by
\lbeq{e.pro}
\mathrm{P}_{C}(y) := \argmin_{x \in C}~ \frac{1}{2} \|x-y\|^2.
\eeq
The proximal-like operator $\mathrm{prox}_{\lambda \wt f}^C(y)$ is the unique optimizer of the optimization problem
\lbeq{e.pro}
\mathrm{prox}_{\lambda \wt f}^C(y) := \argmin_{x \in C}~ \frac{1}{2} \|x-y\|_2^2 + \lambda \wt f(x),
\eeq
where $\lambda>0$. From (\ref{e.fopt}), the first-order optimality condition for the problem (\ref{e.pro}) is given by
\lbeq{e.opts1}
0 \in x-y + \lambda \partial \wt f(x) + N_C(x).
\eeq
If $C=V$, then (\ref{e.opts1}) is simplified to 
\lbeq{e.opts}
0 \in x-y + \lambda \partial \wt f(x),
\eeq
giving the classical proximity operator.

Let $\omega:V\rightarrow\mathbb{R}$ be a differentiable $1$-strongly convex
function, i.e.,
\begin{equation}\label{e.stro}
\omega(y)\geq \omega(x)+\langle \nabla \omega(x), y-x\rangle+
\frac{1}{2}\|y-x\|^2.
\end{equation}
It is assumed that $\omega(x)$ attains its unique minimizer at $x_0$ and
$\omega(x_0)=0$. The function $\omega$ satisfied these conditions is called a prox-function. The corresponding Bregman distance is defined by
\begin{equation}\label{e.breg}
B_\omega(x,y) := \omega(x)-\omega(y)-\langle\nabla \omega(y), x-y\rangle,
\end{equation}
where, from (\ref{e.stro}), it is straightforward to show 
\begin{equation}\label{e.inbreg}
B_\omega(x,y) \geq \frac{1}{2}\|x-y\|^2.
\end{equation}

\section{Single-subproblem accelerated (sub)gradient methods} \label{e.moth}
In this section we first give two schemes for solving structured problems of
the form (\ref{e.gfun}) attaining the optimal complexity for smooth,
nonsmooth, weakly smooth, and smooth strongly problems. These schemes are
optimal up to a logarithmic factor for nonsmooth strongly convex objectives. We then investigate the complexity analysis of these schemes. 

To guarantee the existence of a solution of a problem of the form
(\ref{e.gfun}), we assume:\bigskip\\
{\bf (H1)} The upper level set $N_h(x_0) := \{x \in C \mid h(x) \leq h(x_0)\}$
is bounded, for a starting point $x_0\in C$.\bigskip\\
Since $h$ is convex and $N_h(x_0)$ is closed, (H1) implies that $N_h(x_0)$ is
convex and compact. It therefore follows from the continuity and properness of the objective function $h$ that it attains its global minimizer on $N_h(x_0)$. This guarantees that there is at least one minimizer $x^*$. 

Motivated by {\sc Nesterov} \cite{NesU}, we define  
\lbeq{e.holc}
L_{\nu} := \sup_{x,y\in C,~ x \neq y}\frac{\| \nabla f(x)-\nabla f(y)\|}{\|x-y\|^{\nu}} ,
\eeq
for the level of smoothness parameter $\nu \in [0,1]$. If $L_{\nu}<+\infty$, then (\ref{e.holc}) implies that (\ref{e.holder}) holds resulting to
\lbeq{}
f(x) \leq f(y)+\langle \nabla f(y),x-y\rangle+\frac{L_{\nu}}{1+\nu}\|x-y\|^{1+\nu}
~~~\forall x,y \in C.
\eeq

The following proposition is crucial for constructing our accelerated
(sub)gradient schemes, where for the sake of simplicity for $\nu=1$, we suppose $0^0=1$ in which $\wt L=L_\nu$ is a Lipschitz constant.

\begin{prop}\label{p.norm2}
\cite[Lemma 2]{NesU} Let function $f$ satisfies the condition (\ref{e.holder}). Then, for $\delta>0$
and
\begin{equation}\label{e.lnu}
\widehat{L} \geq \left(\frac{1-\nu}{\delta(1+\nu)} \right)^{\frac{1-\nu}{1+\nu}} L_{\nu}^{\frac{2}{1+\nu}}:=\widetilde{L},
\end{equation}
we have
\begin{equation}\label{e.norm2}
f(x) \leq f(y)+\langle \nabla f(y),x-y\rangle+\frac{1}{2}\widehat{L}\|x-y\|^2
+\frac{\delta}{2}~~~ x, y \in C.
\end{equation}
\end{prop}

The idea is to generate a sequence of estimation functions $\{\phi_k(x)\}_{k \geq 0}$ of $h$ in such a way that, at each iteration $k\geq 0$, the inequality
\begin{equation}\label{e.ineq}
S_k \left(h(x_k)-\frac{\eps}{2} \right) \leq \phi_k^*:=\min_{x\in C}~ \phi_k(x)
\end{equation}
holds for $x_k\in V$, where $S_k$ is a scaling parameter. We consider the
sequence of scaling parameters $\{S_k\}_{k\geq 0}$, which is generated by
\begin{equation}\label{e.sk}
S_k:=S_{k-1}+s_k,
\end{equation}
where $S_0=0$ and $s_k > 0$. We consider the estimation sequence 
\begin{equation}\label{e.phik}
\begin{array}{l}
\phi_{k+1}(x):= \left\{
\begin{array}{ll}
B_\omega(x,x_0) &~~ \mathrm{if}~k=0,\vspace{3mm}\\
\phi_k(x)+ s_{k+1}\left[q_{k+1}(x,y_k)+\psi(x)\right] &~~ \mathrm{if}~k\in \mathbb{N},\\
\end{array}
\right.\vspace{3mm}\\
\D q_{k+1}(x,y_k):=f(y_k)+\langle \nabla f(y_k), x-y_k \rangle+\frac{\mu_f}{2}
\|x-y_k\|^2.
\end{array}
\end{equation}
Let us define $\{z_k\}_{k\geq 0}$ as the sequence of minimizers of the  estimation sequence $\{\phi_k\}_{k\geq 0}$, i.e.,
\begin{equation}\label{e.zk}
z_{k+1} := \D \argmin_{x\in C}~ \phi_{k+1}(x).
\end{equation}

The next result is crucial for the complexity analysis and for providing a stopping  criterion for schemes will be presented in Section \ref{s.algor}.

\begin{prop}\label{p.inphik}
Let the sequence $\{\phi_k\}_{k\geq 0}$ be generated by (\ref{e.phik}). Then
\begin{equation}\label{e.inphifk}
\phi_k(x) \leq S_k~ h(x) + B_\omega(x,x_0)~~~ \forall k\geq 0.
\end{equation}
If in addition (\ref{e.ineq}) holds, then
\begin{equation}\label{e.comp}
h(x_k)-h(x^*) \leq \frac{B_\omega(x^*,x_0)}{S_k}+\frac{\eps}{2}.
\end{equation}
\end{prop}

\begin{proof}
The proof is given by induction on $k$. Since $S_0=0$ and 
$\phi_0(x)=B_\omega(x,x_0)$, the result is valid for $k=0$. We assume it is
true for $k$ and prove it for $k+1$. By this assumption and (\ref{e.phik}), 
we get
\[
\begin{split}
\phi_{k+1}(x)&= \phi_k(x)+ s_{k+1}\left(f(y_k)+\langle \nabla f(y_k), x-y_k \rangle+\frac{\mu_f}{2} \|x-y_k\|^2+\psi(x)\right)\\
&\leq S_k h(x) + B_\omega(x,x_0)+ s_{k+1}\left(f(y_k)+\langle \nabla f(y_k), x-y_k \rangle+\frac{\mu_f}{2} \|x-y_k\|^2+\psi(x)\right)\\
&\leq S_k h(x) + B_\omega(x,x_0)+ s_{k+1} h(x) = S_{k+1} h(x) + B_\omega(x,x_0).
\end{split}
\]

From (\ref{e.ineq}) and (\ref{e.inphifk}), we obtain
\[
h(x_k) \leq \frac{\eps}{2}+\frac{1}{S_k} \phi_k^* \leq 
\frac{\eps}{2}+\frac{1}{S_k} \min_{x\in C} (S_k~ h(x) + B_\omega(x,x_0)) =
\frac{\eps}{2}+h(x^*)+\frac{B_\omega(x^*,x_0)}{S_k},
\]
completing the proof. \qed
\end{proof}

\subsection{{\bf Novel single-subproblem algorithms}} \label{s.algor}
We here give two new algorithms using the estimation sequence (\ref{e.phik})
and investigate the related convergence analysis.

The following result shows that how (\ref{e.phik}) can be used to construct 
the sequence $\{x_k\}_{k\geq 0}$ guaranteeing the condition (\ref{e.ineq}).

\begin{thm}\label{t.alg1}
Let $f$ satisfies (\ref{e.holder}) with $L_{\nu} < +\infty$ and
$\alpha_k:=(s_{k+1}/S_{k+1}) \in {]0,1]}$ for $s_{k+1}>0$. Let also the sequence $\{z_k\}_{k\geq 0}$ be generated by (\ref{e.zk}),
\begin{equation}\label{e.yk1}
y_k := (1-\alpha_k)x_k+\alpha_k z_k,
\end{equation}
\begin{equation}\label{e.xk1}
x_{k+1} := (1-\alpha_k)x_k+\alpha_k z_{k+1},
\end{equation}
and (\ref{e.norm2}) holds for $x=x_{k+1}$, $y=y_k$, $\delta:=\eps \alpha_k$ with $\eps>0$. We set
\begin{equation}\label{e.lhat1}
\widehat{L}_{k+1}:=\left( \frac{1-\nu}{\eps \alpha_k(1+\nu)} \right)^{\frac{1-\nu}{1+\nu}} L_{\nu}^{\frac{2}{1+\nu}}.
\end{equation}
Then we have
\begin{equation}\label{e.inqks}
\phi_{k+1}^* \geq S_{k+1} \left(h(x_{k+1})-\frac{\eps}{2}\right),
\end{equation}
if $s_{k+1}^2 \widehat{L}_{k+1}=(1+S_k\mu)S_{k+1}$ with $\mu= \mu_f+\mu_p$.
\end{thm}

\begin{proof}
The proof is given by induction. Since $S_0=0$, the result for $k=0$ is
evident. Assume that (\ref{e.inqks}) holds for some $k$, and we show that is valid for $k+1$. 

Let us expand $\phi_k$, i.e.,
\lbeq{e.phik11}
\phi_k(x)= 
B_\omega(x,x_0)+\sum_{i=1}^{k} s_i q_i(x,y_{i-1})+S_k \psi(x).
\eeq 
Since $\psi$ is $\mu_p$-strongly convex, (\ref{e.phik11}) implies that $\phi_k$ is $(1+S_k\mu)$-strongly convex. This
and (\ref{e.phik}) at $z_k$ yield 
\begin{equation}\label{e.eine}
\phi_k(x)\geq \phi_k^*+\frac{1+S_k\mu}{2} \|x-z_k\|^2~~~ \forall x\in C.
\end{equation}
From the induction assumption and the convexity of $f$, we obtain
\begin{equation}\label{e.inqks1}
\phi_k^* \geq S_k\left(h(x_k)-\frac{\eps}{2} \right) \geq 
S_k\left(f(y_k)+\langle\nabla f(y_k), x_k-y_k\rangle+\psi(x_k)-\frac{\eps}{2} \right).
\end{equation}
The definition of $y_k$ given in (\ref{e.yk1}) leads to
\begin{equation}\label{e.sim}
\begin{split}
S_k(x_k-y_k)+s_{k+1}(z_{k+1}-y_k) &= S_k x_k-S_{k+1}y_k+s_{k+1}z_{k+1}\\
&= S_k x_k-S_{k+1}((1-\alpha_k)x_k+\alpha_k z_k)+s_{k+1}z_{k+1}
= s_{k+1}(z_{k+1}-z_k).
\end{split}
\end{equation}
By this, (\ref{e.phik}), (\ref{e.eine}), (\ref{e.inqks1}), and (\ref{e.sim}), one can write
\begin{equation}\label{e.phik3}
\begin{split}
\phi_{k+1}^*&\geq \phi_k(z_{k+1})+s_{k+1} \left(f(y_k)+
\langle \nabla f(y_k), z_{k+1}-y_k \rangle +\psi(z_{k+1})\right)\\
&\geq \phi_k^*+\frac{1+S_k\mu}{2} \|z_{k+1}-z_k\|^2+ s_{k+1} \left(f(y_k)+
\langle \nabla f(y_k), z_{k+1}-y_k \rangle +\psi(z_{k+1})\right)\\
&\geq S_k\left(f(y_k)+\langle \nabla f(y_k), x_k-y_k \rangle+\psi(x_k) -\frac{\eps}{2} \right)+\frac{1+S_k\mu}{2} \|z_{k+1}-z_k\|^2\\
&+s_{k+1}\left[f(y_k)+ \langle \nabla f(y_k), z_{k+1}-y_k \rangle +\psi(z_{k+1})\right]\\
&= S_{k+1} f(y_k)+S_k \psi(x_k)+s_{k+1}\psi(z_{k+1})-S_k\frac{\eps}{2}
+\frac{1+S_k\mu}{2} \|z_{k+1}-z_k\|^2 \\
&+\langle \nabla f(y_k), S_k(x_k-y_k)+s_{k+1}(z_{k+1}-y_k) \rangle\\
&= S_{k+1} f(y_k)+S_k \psi(x_k)+s_{k+1}\psi(z_{k+1})-S_k\frac{\eps}{2}
+\frac{1+S_k\mu}{2}\|z_{k+1}-z_k\|^2\\
&+ s_{k+1}\langle \nabla f(y_k), z_{k+1}-z_k \rangle.
\end{split} 
\end{equation}
By the convexity of $\psi$ and (\ref{e.xk1}), we get
\begin{equation}\label{e.psiin}
S_k \psi(x_k)+s_{k+1}\psi(z_{k+1})=S_{k+1}(\alpha_k \psi(z_{k+1})+(1-\alpha_k)
\psi(x_k))\geq S_{k+1} \psi(x_{k+1}).
\end{equation}
The definitions of $y_k$ and $x_{k+1}$ yield
\[
x_{k+1}-y_k= (1-\alpha_k)x_k+\alpha_k z_{k+1}- (1-\alpha_k)x_k-\alpha_k z_k
= \alpha_k (z_{k+1}-z_k).
\]
From this, (\ref{e.phik3}), and (\ref{e.psiin}), we obtain
\begin{equation}\label{e.phik4}
\begin{split}
\phi_{k+1}^* &\geq  S_{k+1} f(y_k)+S_{k+1} \psi(x_{k+1})
+S_{k+1}\langle \nabla f(y_k), x_{k+1}-y_k \rangle-S_k\frac{\eps}{2}+\frac{1+S_k\mu}{2} \|z_{k+1}-z_k\|^2.
\end{split} 
\end{equation}
By (\ref{e.norm2}) for $\delta=\alpha_k \eps$, we get
\[
f(y_k)+\langle\nabla f(y_k),x_{k+1}-y_k \rangle \geq f(x_{k+1})
-\frac{\widehat{L}_{k+1}}{2} \|x_{k+1}-y_k\|^2-\frac{\alpha_k \eps}{2}.
\]
It follows from this, (\ref{e.inbreg}), and (\ref{e.phik4}) that
\[
\begin{split}
\phi_{k+1}^* &\geq S_{k+1} \left(f(y_k)+\langle \nabla f(y_k), x_{k+1}-y_k \rangle\right)
+S_{k+1} \psi(x_{k+1})-S_k\frac{\eps}{2}+\frac{1+S_k\mu}{2} \|z_{k+1}-z_k\|^2\\
&\geq S_{k+1} \left(f(x_{k+1})-\frac{\widehat{L}_{k+1}}{2} \|x_{k+1}-y_k\|^2-\frac{\alpha_k \eps}{2} \right)
+S_{k+1} \psi(x_{k+1})-S_k\frac{\eps}{2}+\frac{1+S_k\mu}{2\alpha_k^2}\|x_{k+1}-y_k\|^2\\
&= S_{k+1} \left(h(x_{k+1})-\frac{\eps}{2} \right)
+\frac{1}{2}\frac{S_{k+1}}{s_{k+1}^2}\left( (1+S_k\mu)S_{k+1}-s_{k+1}^2 \widehat{L}_{k+1} \right) \|x_{k+1}-y_k\|^2.
\end{split} 
\]
Therefore, $s_{k+1}^2 \widehat{L}_{k+1}=(1+S_k\mu)S_{k+1}$ implies that (\ref{e.inqks})
holds. \qed
\end{proof}

Let us assume that $\widehat{L}_{k+1}$ is given. Then $s_{k+1}$ is given by the
positive solution of the equation $s_{k+1}^2 \widehat{L}_{k+1}=(1+S_k\mu)S_{k+1}$, i.e.,
\begin{equation}\label{e.skp1}
s_{k+1}=\frac{1+S_k\mu+((1+S_k\mu)^2+4\widehat{L}_{k+1}S_k(1+S_k\mu))^{1/2}}{2\widehat{L}_{k+1}}>0.
\end{equation}
Indeed, Theorem \ref{t.alg1} leads to a simple scheme for solving problems of
the form (\ref{e.gfun}). We summarize this scheme in the following. 

\vspace{4mm}
\begin{algorithm}[H] 
\DontPrintSemicolon 
\KwIn{ initial point $x_0 \in C$,~$\nu$,~$L_\nu$,~$\mu\geq 0$;~ $\eps>0$;}
\KwOut{$x_k$,~ $h_k$;}
\Begin{
    \While {stopping criteria do not hold}{
        compute $\widehat{L}_{k+1}$;~
        compute $s_{k+1}$ by (\ref{e.skp1});~$S_{k+1} = S_k+s_{k+1}$;~ 
        $\alpha_k = s_{k+1}/S_{k+1}$;\;
        $y_k = \alpha_k z_k+(1-\alpha_k) x_k$;~
        compute $z_{k+1}$ from (\ref{e.zk});~
        $x_{k+1}=(1-\alpha_k)x_k+\alpha_k z_{k+1}$;      
        $k=k+1$;\;
    }
    $h_k=h(x_k)$;
}
\caption{ {\bf ASGA-1} (single-subproblem ASGA)}
\label{a.aga} 
\end{algorithm}

\vspace{5mm}
ASGA-1 has a simple structure and each iteration needs only a solution of the
auxiliary problem (\ref{e.zk}) (Line 4), i.e., only one call of the oracle is needed per each iteration. Let us denote by $N(k)$ the total number of calls of
the first-order oracle after $k$ iterations. Therefore, we have that $N(k)=k$
for ASGA-1. 

For implementation of ASGA-1 one needs to know about $\widehat{L}_{k+1}$ in
each step. The next result shows how to compute $\widehat{L}_{k+1}$ if the
parameters $\nu$ and $L_\nu$ are available.

\begin{prop}\label{p.lhat}
Let $\{y_k\}_{k\geq 0}$, $\{z_k\}_{k\geq 0}$, and $\{x_k\}_{k\geq 0}$ be
generated by ASGA-1 and $s_{k+1}^2 \widehat{L}_{k+1}=(1+S_k\mu)S_{k+1}$.
Then $\widehat{L}_{k+1}$ can be computed by solving the one-dimensional 
nonlinear equation 
\begin{equation}\label{e.nequ}
\widehat{L}_{k+1}-\left( 1+S_k\mu+((1+S_k\mu)^2+4\widehat{L}_{k+1}S_k(1+S_k\mu))^{1/2}
\right)^{\frac{1-\nu}{1+\nu}} \widetilde{L}_{k+1}=0,
\end{equation}
where
\begin{equation}\label{e.ltilde}
\widetilde{L}_{k+1}:=\left( \frac{1-\nu}{2(1+S_k\mu)\eps(1+\nu)} \right)^{\frac{1-\nu}{1+\nu}} L_{\nu}^{\frac{2}{1+\nu}}.
\end{equation}
\end{prop}

\begin{proof}
The solution of $s_{k+1}^2 \widehat{L}_{k+1}=(1+S_k\mu)S_{k+1}$ is given by
(\ref{e.skp1}). The definition of $\alpha_k$ and dividing both sides of 
$s_{k+1}^2 \widehat{L}_{k+1}=(1+S_k\mu)S_{k+1}$ by $S_{k+1}$ yield
$\alpha_k=(1+S_k\mu)/s_{k+1}\widehat{L}_{k+1}$. Substituting (\ref{e.skp1}) into this equation
gives
\[
\alpha_k=2(1+S_k\mu)/\left(1+S_k\mu+((1+S_k\mu)^2+
4\widehat{L}_{k+1}S_k(1+S_k\mu))^{1/2}\right).
\]
By substituting this into (\ref{e.lhat1}) with $\delta=\alpha_k \eps$, we get
\[
\widehat{L}_{k+1}= \left( 1+S_k\mu+((1+S_k\mu)^2+4\widehat{L}_{k+1}S_k(1+S_k\mu))^{1/2}
\right)^{\frac{1-\nu}{1+\nu}} \left( \frac{1-\nu}{2(1+S_k\mu)\eps(1+\nu)} \right)^{\frac{1-\nu}{1+\nu}} L_{\nu}^{\frac{2}{1+\nu}},
\]
giving (\ref{e.nequ}) where $\widetilde{L}_{k+1}$ is given by (\ref{e.ltilde}). It remains to show that the equation (\ref{e.nequ}) has a solution. Let us define
$\zeta:\mathbb{R} \rightarrow\mathbb{R}$ by
\[
\zeta(\theta):=\theta-\left( 1+S_k\mu+((1+S_k\mu)^2+4\theta S_k(1+S_k\mu))^{1/2} \right)^{\frac{1-\nu}{1+\nu}} \widetilde{L}_{k+1}.
\]
Since $\widetilde{L}_{k+1}>0$, we get $\zeta(0)<0$. We also have
\[
\lim_{\theta\rightarrow \infty} \theta/\left( 1+S_k\mu+ ((1+S_k\mu)^2
+4\theta S_k(1+S_k\mu))^{1/2} \right)^{\frac{1-\nu}{1+\nu}} \widetilde{L}_{k+1}=+\infty,
\] 
implying there exists $\theta_1>0$ such that for $\theta>\theta_1$ we have 
\[
\theta > \left( 1+S_k\mu+ ((1+S_k\mu)^2 +4\theta S_k(1+S_k\mu))^{1/2}
\right)^{\frac{1-\nu}{1+\nu}} \widetilde{L}_{k+1}.
\]
This implies that for
$\theta>\theta_1$ we have $\zeta(\theta)>0$. Therefore, the equation
(\ref{e.nequ}) has a solution.\qed
\end{proof}

In view of Proposition \ref{p.lhat}, if $\nu$ and $L_\nu$ are available, one
can compute $\widehat{L}_{k+1}$ by solving the one-dimensional nonlinear
equation (\ref{e.nequ}). If one solves the equation $\zeta(\theta)=0$ approximately, and an initial interval $[a,b]$ is available such that $\varphi(a)\varphi(b)<0$, then a solution can be computed to $\eps$-accuracy using the bisection scheme in $\mathcal{O}(\log_2((b-a)/\varepsilon))$ iterations, see, e.g., \cite{NeuB}. However, it is preferable to use a more sophisticated zero finder like the secant bisection scheme (Algorithm 5.2.6, \cite{NeuB}). For solving this nonlinear equation, one can also take advantage
of MATLAB $\mathtt{fzero}$ function combining the bisection scheme, the
inverse quadratic interpolation, and the secant method. On the other hand, if
$\nu$ and $L_\nu$ are not available, ASGA-1 cannot be used directly, which is
the case in many black-box optimization problems. 

The subsequent result gives the complexity of ASGA-1 for attaining an 
$\eps$-solution of (\ref{e.gfun}).

\begin{thm}\label{t.compl1}
Let $\{x_k\}_{k\geq 0}$ be generated by ASGA-1. Then \\\\
(i) If $\mu>0$, we have
\begin{equation}\label{e.comb2}
h(x_k)-h(x^*) \leq \widehat{L}_1 \left(1+\frac{\mu^{\frac{1+\nu}{1+3\nu}}\eps^{\frac{1-\nu}{1+3\nu}}}{2 L_\nu^{\frac{2}{1+3\nu}}} 
\right)^{-\frac{1+3\nu}{1+\nu}(k-1)} B_\omega(x^*,x_0) + \frac{\eps}{2},
\end{equation}
where
\begin{equation}\label{e.lh0}
\widehat{L}_1=\left( \frac{1-\nu}{\eps(1+\nu)} \right)^{\frac{1-\nu}{1+\nu}}
L_{\nu}^{\frac{2}{1+\nu}}.
\end{equation} 
(ii) If $\mu=0$, we have
\begin{equation}\label{e.comb22}
h(x_k)-h(x^*) \leq \left(\frac{2^{\frac{1+3\nu}{1+\nu}}L_{\nu}^{\frac{2}{1+\nu}}}{\eps^{\frac{1-\nu}{1+\nu}} k ^{\frac{1+3\nu}{1+\nu}}}
\right) B_\omega(x^*,x_0) + \frac{\eps}{2}.
\end{equation} 
\end{thm}

\begin{proof}
(i) By (\ref{e.lhat1}), $s_k^2 \widehat{L}_k=(1+S_{k-1}\mu)S_k$, $\alpha_{k-1}=s_k/S_k$,
we get
\[
\frac{s_k^2}{S_k} = \frac{1+S_{k-1}\mu}{\widehat{L}_k} 
\geq (1+S_{k-1}\mu)(\eps \alpha_{k-1})^{\frac{1-\nu}{1+\nu}}L_\nu^{-\frac{2}{1+\nu}},
\]
leading to
\[
s_k^2 \geq (1+S_{k-1}\mu) L_\nu^{-\frac{2}{1+\nu}} (\eps s_k)^{\frac{1-\nu}{1+\nu}} S_k^{\frac{2\nu}{1+\nu}}.
\]
This implies
\begin{equation}\label{e.inskp2}
s_k S_k^{-\frac{2\nu}{1+3\nu}} \geq (1+S_{k-1}\mu)^{\frac{1+\nu}{1+3\nu}}
\eps^{\frac{1-\nu}{1+3\nu}} L_\nu^{-\frac{2}{1+3\nu}}.
\end{equation}
It follows from $S_{k+1} \geq S_k$ and (\ref{e.inskp2}) that
\[
\begin{split}
S_k^{\frac{1+\nu}{1+3\nu}}-S_{k-1}^{\frac{1+\nu}{1+3\nu}} &\geq 
(S_k-S_{k-1})/\left(S_k^{1-\frac{1+\nu}{1+3\nu}}-S_{k-1}^{1-\frac{1+\nu}{1+3\nu}} \right) \geq \frac{1}{2} s_k S_k^{-\frac{2\nu}{1+3\nu}}\\
&\geq 2^{-1} (1+S_{k-1}\mu)^{\frac{1+\nu}{1+3\nu}}\eps^{\frac{1-\nu}{1+3\nu}}
L_\nu^{-\frac{2}{1+3\nu}} \geq 2^{-1} 
(S_{k-1}\mu)^{\frac{1+\nu}{1+3\nu}}\eps^{\frac{1-\nu}{1+3\nu}}
L_\nu^{-\frac{2}{1+3\nu}}.
\end{split}
\]
By $S_0=0$ and (\ref{e.nequ}), we have  $S_1=\widehat{L}_1^{-1}$, where
$L_0$ is given by (\ref{e.lh0}). Hence we have
\[
S_k^{\frac{1+\nu}{1+3\nu}}\geq \left(1+2^{-1} 
\mu^{\frac{1+\nu}{1+3\nu}}\eps^{\frac{1-\nu}{1+3\nu}}
L_\nu^{-\frac{2}{1+3\nu}}\right) S_{k-1}^{\frac{1+\nu}{1+3\nu}}
\geq \dots \geq \left(1+2^{-1} 
\mu^{\frac{1+\nu}{1+3\nu}}\eps^{\frac{1-\nu}{1+3\nu}}
L_\nu^{-\frac{2}{1+3\nu}}\right)^{k-1} S_1^{\frac{1+\nu}{1+3\nu}},
\]
leading to
\[
S_k\geq \widehat{L}_1^{-1} \left(1+2^{-1} \mu^{\frac{1+\nu}{1+3\nu}}\eps^{\frac{1-\nu}{1+3\nu}}
L_\nu^{-\frac{2}{1+3\nu}}\right)^{\frac{1+3\nu}{1+\nu}(k-1)}.
\]
This inequality and (\ref{e.comp}) give (\ref{e.comb2}). 

(ii) Substituting $\mu=0$ into (\ref{e.inskp2}) yields 
\[
s_k S_k^{-\frac{2\nu}{1+3\nu}} \geq \eps^{\frac{1-\nu}{1+3\nu}}
L_\nu^{-\frac{2}{1+3\nu}}.
\]
It follows from $S_k \geq S_{k-1}$ and (\ref{e.inskp2}) that
\[
S_k^{\frac{1+\nu}{1+3\nu}}-S_{k-1}^{\frac{1+\nu}{1+3\nu}} \geq 
(S_k-S_{k-1})/\left(S_k^{1-\frac{1+\nu}{1+3\nu}}-S_{k-1}^{1-\frac{1+\nu}{1+3\nu}} \right) \geq \frac{1}{2} s_k S_k^{-\frac{2\nu}{1+3\nu}}
\geq 2^{-1} \eps^{\frac{1-\nu}{1+3\nu}}
L_\nu^{-\frac{2}{1+3\nu}}.
\]
Let us sum up this inequality for $i=0,\dots,k$, giving
\[
S_k^{\frac{1+\nu}{1+3\nu}} \geq k 2^{-1} \eps^{\frac{1-\nu}{1+3\nu}}
L_\nu^{-\frac{2}{1+3\nu}},
\]
leading to
\[
S_k \geq k^{\frac{1+3\nu}{1+\nu}} \eps^{\frac{1-\nu}{1+\nu}}
2^{-\frac{1+3\nu}{1+\nu}} L_\nu^{-\frac{2}{1+\nu}}.
\]
This inequality and (\ref{e.comp}) give (\ref{e.comb22}). \qed
\end{proof}

The next result gives the complexity of ASGA-1 for giving an $\eps$-solution of
the problem (\ref{e.gfun}).

\begin{cor}\label{c.comp1}
Let $\{x_k\}_{k\geq 0}$ be generated by ASGA-1. Then \\\\
(i) If $\mu>0$, then an $\eps$-solution of the problem (\ref{e.gfun}) is given by the complexity
\begin{equation}\label{e.compl1}
\mathcal{O}\left( \mu^{-\frac{1+\nu}{1+3\nu}} L_\nu^{\frac{2}{1+3\nu}}
\eps^{-\frac{1-\nu}{1+3\nu}} \ln(\eps^{-\frac{2}{1+\nu}}) \right).
\end{equation}
(ii) If $\mu=0$, then an $\eps$-solution of the problem (\ref{e.gfun}) is given by the complexity
\begin{equation}\label{e.compl2}
\mathcal{O}\left(\eps^{-\frac{2}{1+3\nu}} \right).
\end{equation}
\end{cor}

\begin{proof}
From the right hand side of (\ref{e.comb2}), we obtain
\[
2\left(\frac{1-\nu}{1+\nu}\right)^{\frac{1-\nu}{1+\nu}}
L_{\nu}^{\frac{2}{1+\nu}} \left(1+\frac{\mu^{\frac{1+\nu}{1+3\nu}}\eps^{\frac{1-\nu}{1+3\nu}}}{2 L_\nu^{\frac{2}{1+3\nu}}} 
\right)^{-\frac{1+3\nu}{1+\nu}(k-1)} B_\omega(x^*,x_0) \leq 
\eps^{\frac{2}{1+\nu}}, 
\]
leading to
\[
\ln(A_1)-\ln(\eps^{\frac{2}{1+\nu}}) \leq \frac{1+3\nu}{1+\nu}(k-1)
~\ln\left(1+2^{-1}\mu^{\frac{1+\nu}{1+3\nu}}\eps^{\frac{1-\nu}{1+3\nu}}
L_\nu^{-\frac{2}{1+3\nu}}\right)
 \leq \frac{1+3\nu}{2(1+\nu)}~\mu^{\frac{1+\nu}{1+3\nu}}
\eps^{\frac{1-\nu}{1+3\nu}} L_\nu^{-\frac{2}{1+3\nu}}(k-1),
\]
where
\[
A_1:=2\left(\frac{1-\nu}{1+\nu}\right)^{\frac{1-\nu}{1+\nu}}
L_{\nu}^{\frac{2}{1+\nu}} B_\omega(x^*,x_0).
\]
This yields
\[
k \geq \frac{2(1+\nu)}{1+3\nu}~\mu^{-\frac{1+\nu}{1+3\nu}}
\eps^{-\frac{1-\nu}{1+3\nu}} L_\nu^{\frac{2}{1+3\nu}}
\left( \ln(A_1)+\ln(\eps^{-\frac{2}{1+\nu}}) \right),
\]
implying that (\ref{e.compl1}) is valid.

By (\ref{e.comb22}), we get
\[
2^{\frac{1+3\nu}{1+\nu}} L_\nu^{\frac{2}{1+\nu}} k^{-\frac{1+3\nu}{1+\nu}} 
\eps^{-\frac{1-\nu}{1+\nu}} B_\omega(x^*,x_0) + \frac{\eps}{2} \leq \eps,
\]
leading to
\[
k \geq 2^{\frac{2+4\nu}{1+3\nu}} L_\nu^{\frac{2}{1+3\nu}}
\eps^{-\frac{2}{1+3\nu}} B_\omega(x^*,x_0)^{\frac{1+\nu}{1+3\nu}},
\]
implying that (\ref{e.compl2}) is valid. \qed
\end{proof}

In the remainder of this section we give a way to get rid of needing the 
parameters $\nu$ and $L_\nu$ using a backtracking line search guaranteeing
(\ref{e.norm2}). This leads to a parameter-free version of ASGA-1 given in the
next result (see Algorithm \ref{a.aga2}, ASGA-2).

\begin{thm}\label{t.alg}
Let $f$ satisfies (\ref{e.holder}) with $L_{\nu} < +\infty$.
Let $\alpha_k:=(s_{k+1}/S_{k+1}) \in {]0,1]}$ for $s_{k+1}>0$, the sequence $\{z_k\}_{k\geq 0}$, $\{y_k\}_{k\geq 0}$, and $\{x_k\}_{k\geq 0}$ be generated
by (\ref{e.zk}), (\ref{e.yk1}), and (\ref{e.xk1}), respectively, such that
\begin{equation}\label{e.blin}
f(x_{k+1})\leq f(y_k)+ \langle\nabla f(y_k),x_{k+1}-y_k\rangle+                \frac{L_{k+1}}{2} \|x_{k+1} - y_k\|^2+\frac{\alpha_k\eps}{2},
\end{equation} 
for $L_{k+1}\geq \widetilde{L}>0$. Then (\ref{e.inqks}) holds if 
$s_{k+1}^2 L_{k+1}=(1+S_k\mu)S_{k+1}$.
\end{thm}

\begin{proof}
Following the proof of Theorem \ref{t.alg1}, the inequality (\ref{e.phik4})
is valid. From (\ref{e.blin}), for $\delta=\alpha_k \eps$, we obtain
\[
f(y_k)+\langle\nabla f(y_k),x_{k+1}-y_k \rangle \geq f(x_{k+1})
-\frac{L_{k+1}}{2} \|x_{k+1}-y_k\|^2-\frac{\alpha_k\eps}{2}.
\]
By this, (\ref{e.inbreg}), and (\ref{e.phik4}), we can write
\[
\begin{split}
\phi_{k+1}^* &\geq S_{k+1} f(y_k)+S_{k+1} \psi(x_{k+1})
+S_{k+1}\langle \nabla f(y_k), x_{k+1}-y_k \rangle-S_k\frac{\eps}{2}+\frac{1+S_k\mu}{2} \|z_{k+1}-z_k\|^2\\
&\geq S_{k+1} \left(f(y_k)+\langle \nabla f(y_k), x_{k+1}-y_k \rangle\right)
+S_{k+1} \psi(x_{k+1})-S_k\frac{\eps}{2}+\frac{1+S_k\mu}{2} \|z_{k+1}-z_k\|^2\\
&\geq S_{k+1} \left(f(x_{k+1})-\frac{L_{k+1}}{2} \|x_{k+1}-y_k\|^2-\frac{\eps}{2} \alpha_k \right)
+S_{k+1} \psi(x_{k+1})-S_k\frac{\eps}{2}+\frac{1+S_k\mu}{2\alpha_k^2}
\|x_{k+1}-y_k\|^2\\
&= S_{k+1} \left(h(x_{k+1})-\frac{\eps}{2} \right)
+\frac{1}{2}\frac{S_{k+1}}{s_{k+1}^2}\left( (1+S_k\mu)S_{k+1}-s_{k+1}^2 L_{k+1} \right) \|x_{k+1}-y_k\|^2.
\end{split} 
\]
Therefore, setting $s_{k+1}^2 L_{k+1}=(1+S_k\mu)S_{k+1}$ yields that 
(\ref{e.inqks}) holds. \qed
\end{proof}

To guarantee the inequality (\ref{e.blin}), we
assume $L_0>0$ and set $L_{k+1}:=\gamma_2 \gamma_1^{p_k} L_k$, for $p_k\geq 0$,
$\gamma_1>1$ and $\gamma_2<1$, such that $L_{k+1}\geq \widetilde{L}$
guaranteeing that (\ref{e.norm2}) holds for $\delta=\eps \alpha_k$. We give the
detailed results in the next proposition.

\begin{prop}\label{p.weld}
Let $\{z_k\}_{k\geq 0}$, $\{y_k\}_{k\geq 0}$, and $\{x_k\}_{k\geq 0}$ be
generated by (\ref{e.zk}), (\ref{e.yk1}), and (\ref{e.xk1}), respectively. 
Let also $L_0>0$ and 
\[
\overline{L}_{k+1}:= \gamma_1^{p_k} L_k,~~~ s_{k+1}^2 L_{k+1}=(1+S_k\mu)S_{k+1},
\]
for $p_k \geq 0$. Then $s_{k+1}>0$ and for
\begin{equation}\label{e.pk}
p_k \geq \frac{1-\nu}{1+\nu} \log_{\gamma_1} \left( \frac{1-\nu}{\alpha_k\eps(1+\nu)} \right) + \frac{2}{1+\nu} \log_{\gamma_1} L_{\nu}
-\log_{\gamma_1} {L_k} 
\end{equation}
the inequality (\ref{e.blin}) is satisfied. 
\end{prop}

\begin{proof}
By $L_0>0$ and $\overline{L}_{k+1}= \gamma_1^{p_k} L_k$, we have
$\overline{}L_{k+1}>0$. The solution of the equation
\[
\overline{L}_{k+1} s_{k+1}^2-(1+S_k\mu)s_{k+1}-(1+S_k\mu)S_k=0
\]
is given by
\begin{equation}\label{e.skp2}
s_{k+1}=\frac{1+S_k\mu+((1+S_k\mu)^2+4\overline{L}_{k+1}S_k(1+S_k\mu))^{1/2}}{2\overline{L}_{k+1}}>0.
\end{equation}
By setting $\delta:=\alpha_k \eps$, Proposition \ref{p.norm2} suggests
that if $\overline{L}_{k+1}= \gamma_1^{p_k} L_k \geq \widetilde{L}$, then (\ref{e.blin}) is valid leading to
\[
\overline{L}_{k+1} = \gamma_1^{p_k} L_k\geq \left( \frac{1-\nu}{\delta(1+\nu)} \right)^{\frac{1-\nu}{1+\nu}} L_{\nu}^{\frac{2}{1+\nu}}.
\] 
This implies
\[
p_k \ln \gamma_1 \geq \ln \left( \frac{1-\nu}{\alpha_k \eps(1+\nu)} \right)^{\frac{1-\nu}{1+\nu}} + \ln \left(\frac{L_{\nu}^{\frac{2}{1+\nu}}}{L_k}\right),
\] 
giving (\ref{e.pk}). \qed
\end{proof}

Theorem \ref{t.alg} leads to a simple iterative scheme for solving the problem
(\ref{e.gfun}), where the sequences $\{z_k\}_{k\geq 0}$, $\{y_k\}_{k\geq 0}$,
and $\{x_k\}_{k\geq 0}$ are generated by (\ref{e.zk}), (\ref{e.yk1}), and
(\ref{e.xk1}), respectively. Proposition \ref{p.weld} shows that the condition
(\ref{e.blin}) holds in finite iterations of a backtracking line search. We
summarize the above-mentioned discussion in the following
algorithm:

\vspace{5mm}
\begin{algorithm}[H] 
\DontPrintSemicolon 
\KwIn{ initial point $x_0\in C$, $L_0>0$,~$\gamma_1>1$,~$\gamma_2<1$ ~$p=0$,~$\mu\geq 0$;~$\eps>0$;}
\KwOut{$x_k$,~ $h_k$;}
\Begin{
    \While {stopping criteria do not hold}{
        \Repeat{$f(\widehat{x}_{k+1})\leq f(y_k)+
                \langle\nabla f(y_k),\widehat{x}_{k+1}-y_k\rangle+
               \frac{1}{2}\overline{L}_{k+1} \|\widehat{x}_{k+1} - y_k\|^2
               +\frac{1}{2}\alpha_k\eps$}{
               $\overline{L}_{k+1} = \gamma_1^p L_k$;~ 
               compute $s_{k+1}$ by (\ref{e.skp2});~
               $\widehat{S}_{k+1} = S_k+s_{k+1}$;~ 
               $\alpha_k = s_{k+1}/\widehat{S}_{k+1}$;\;
               $y_k = \alpha_k z_k+(1-\alpha_k) x_k$;
               compute $\widehat{z}_{k+1}$ from (\ref{e.zk}); 
               $\widehat{x}_{k+1}=\alpha_k\widehat{z}_{k+1}+(1-\alpha_k)x_k$;
               $p=p+1$;      
        }        
        $x_{k+1}=\widehat{x}_{k+1}$;~ $z_{k+1}=\widehat{z}_{k+1}$;
        $S_{k+1}=\widehat{S}_{k+1}$;~$L_{k+1}=\gamma_2\overline{L}_{k+1}$;~
        $k=k+1$;~$p=0$;\;
    }
    $h_k=h(x_k)$;
}
\caption{ {\bf ASGA-2} (parameter-free single-subproblem ASGA)}
\label{a.aga2} 
\end{algorithm}

\vspace{5mm}
ASGA-2 in each iteration needs at least a solution of the
auxiliary problem (\ref{e.zk}) until (\ref{e.blin}) holds. The
loop between Line 3 and Line 6 of ASGA-2 is called the inner cycle, and the loop
between Line 2 and Line 8 of ASGA-2 is called the outer cycle. Hence
Proposition \ref{p.weld} shows that the inner cycle is terminated in a finite
number of inner iterations. Since it is not assumed to have
\[
\langle\nabla f(y_k),x_{k+1}-y_k\rangle+ \frac{\overline{L}_{k+1}}{2} \|x_{k+1}-y_k\|^2+
\frac{\eps}{2}\alpha_k \leq 0,
\]
one cannot guarantee the descent condition $f(x_{k+1})\leq f(x_k)$, i.e.,
$h(x_{k+1})\leq h(x_k)$ is not guaranteed. Therefore, the line search  (\ref{e.blin}) is
nonmonotone (see more about nonmonotone line searches in \cite{AhoG,AmiAN} and
references therein). 

We compute the total number of calls of the first-order oracle after $k$ iteration ($N(k)$) for ASGA-2 in the subsequent result.

\begin{prop}\label{p.tnoc}
Let $\{x_k\}_{k\geq 0}$ be generated by ASGA-2. Then 
\begin{equation}\label{e.nk}
N(k) \leq 2 \left( 1-\frac{\ln \gamma_2}{\ln \gamma_1} \right) (k+1)+
\frac{2}{\ln \gamma_1} \ln \frac{\gamma_1\gamma_2 \wt L}{L_0}.
\end{equation}
\end{prop}

\begin{proof}
From $L_{i+1}=\gamma_2 \gamma_1^{p_i} L_i,~ i=0,\dots, k$, we obtain
\[
p_i = \frac{1}{\ln \gamma_1}(\ln L_{i+1}-\ln L_i-\ln \gamma_2).
\]
By this and $L_{k+1} \leq \gamma_1\gamma_2 \wt L$, we get
\[
\begin{split}
N(k)&= \sum_{i=0}^k (2p_i+2) = \sum_{i=0}^k 
\left(\frac{1}{\ln \gamma_1}(\ln L_{i+1}-\ln L_i-\ln \gamma_2)+2\right)\\ 
&=2 \left( 1-\frac{\ln \gamma_2}{\ln \gamma_1} \right) (k+1)+\frac{2}{\ln \gamma_1} \ln \frac{L_{k+1}}{L_0} \leq 2 \left( 1-\frac{\ln \gamma_2}{\ln \gamma_1} \right) (k+1)+\frac{2}{\ln \gamma_1} 
\ln \frac{\gamma_1\gamma_2 \wt L}{L_0},
\end{split}
\]
giving the result.\qed
\end{proof}

Proposition \ref{p.tnoc} implies that ASGA-2 on average requires at least two
calls of the first-order oracle per iteration, whereas ASGA-1 needs a single
call of the first-order oracle per iteration.

We derive the complexity of ASGA-2 in the next result that is slightly
modification of Theorem \ref{t.compl1}.

\begin{thm}\label{t.compl2}
Let $\{x_k\}_{k\geq 0}$ be generated by ASGA-2. Then \\\\
(i) If $\mu>0$, we have
\begin{equation}\label{e.comb1}
h(x_k)-h(x^*) \leq L_1 \left(1+\frac{\mu^{\frac{1+\nu}{1+3\nu}} \eps^{\frac{1-\nu}{1+3\nu}}}{2\gamma_1^{\frac{1+\nu}{1+3\nu}} L_\nu^{\frac{2}{1+3\nu}}} 
\right)^{-\frac{1+3\nu}{1+\nu}(k-1)} B_\omega(x^*,x_0) + \frac{\eps}{2}, 
\end{equation} 
where $L_1=\gamma_2 \gamma_1^{p_1}L_0$.\\
(ii) If $\mu=0$, we have
\begin{equation}\label{e.comb11}
h(x_k)-h(x^*) \leq \left(\frac{\gamma_1 
2^{\frac{1+3\nu}{1+\nu}}L_{\nu}^{\frac{2}{1+\nu}}}{\eps^{\frac{1-\nu}{1+\nu}} k^{\frac{1+3\nu}{1+\nu}}} \right) B_\omega(x^*,x_0) + \frac{\eps}{2}.
\end{equation} 
\end{thm}

\begin{proof}
(i) From Propositions \ref{p.norm2} and \ref{p.weld}, we obtain 
\[
\frac{1}{\gamma_1}L_k=\gamma_1^{p_k-1} L_{k-1} \leq \left( \frac{1-\nu}{\eps \alpha_{k-1} (1+\nu)} \right)^{\frac{1-\nu}{1+\nu}} L_\nu^{\frac{2}{1+\nu}}\leq 
\left( \eps \alpha_{k-1} \right)^{-\frac{1-\nu}{1+\nu}} L_\nu^{\frac{2}{1+\nu}}.
\] 
By this, $s_k^2 L_k=(1+S_{k-1}\mu)S_k$, and $\alpha_{k-1}=s_k/S_k$, we get
\[
\frac{s_k^2}{S_k} = \frac{1+S_{k-1}\mu}{L_k} 
\geq \gamma_1^{-1}(1+S_{k-1}\mu)(\eps \alpha_{k-1})^{\frac{1-\nu}{1+\nu}}L_\nu^{-\frac{2}{1+\nu}},
\]
leading to
\[
s_k^2 \geq \gamma_1^{-1} (1+S_{k-1}\mu) L_\nu^{-\frac{2}{1+\nu}} (\eps s_k)^{\frac{1-\nu}{1+\nu}} S_k^{\frac{2\nu}{1+\nu}}.
\]
This implies
\begin{equation}\label{e.inskp22}
s_k S_k^{-\frac{2\nu}{1+3\nu}} \geq (1+S_{k-1}\mu)^{\frac{1+\nu}{1+3\nu}}
\eps^{\frac{1-\nu}{1+3\nu}} \gamma_1^{-\frac{1+\nu}{1+3\nu}} 
L_\nu^{-\frac{2}{1+3\nu}}.
\end{equation}
It follows from $S_{k+1} \geq S_k$ and (\ref{e.inskp2}) that
\[
\begin{split}
S_k^{\frac{1+\nu}{1+3\nu}}-S_{k-1}^{\frac{1+\nu}{1+3\nu}} &\geq 
(S_k-S_{k-1})/\left(S_k^{1-\frac{1+\nu}{1+3\nu}}-S_{k-1}^{1-\frac{1+\nu}{1+3\nu}} \right) \geq \frac{1}{2} s_k S_k^{-\frac{2\nu}{1+3\nu}}\\
&\geq 2^{-1}(1+S_{k-1}\mu)^{\frac{1+\nu}{1+3\nu}}
\eps^{\frac{1-\nu}{1+3\nu}} \gamma_1^{-\frac{1+\nu}{1+3\nu}} 
L_\nu^{-\frac{2}{1+3\nu}} \geq 2^{-1} (S_{k-1}\mu)^{\frac{1+\nu}{1+3\nu}}
\eps^{\frac{1-\nu}{1+3\nu}} \gamma_1^{-\frac{1+\nu}{1+3\nu}} 
L_\nu^{-\frac{2}{1+3\nu}}.
\end{split}
\]
Then we have
\[
S_k^{\frac{1+\nu}{1+3\nu}}\geq \left(1+2^{-1} 
\mu^{\frac{1+\nu}{1+3\nu}} \eps^{\frac{1-\nu}{1+3\nu}} \gamma_1^{-\frac{1+\nu}{1+3\nu}} 
L_\nu^{-\frac{2}{1+3\nu}}\right) S_{k-1}^{\frac{1+\nu}{1+3\nu}}.
\]
Since $S_0=0$, we have $S_1=L_1^{-1}$ leading to
\[
S_k\geq \left(1+2^{-1} 
\mu^{\frac{1+\nu}{1+3\nu}} \eps^{\frac{1-\nu}{1+3\nu}} \gamma_1^{-\frac{1+\nu}{1+3\nu}} 
L_\nu^{-\frac{2}{1+3\nu}}\right)^{\frac{1+3\nu}{1+\nu}(k-1)} L_1^{-1}.
\]
This inequality and (\ref{e.comp}) give (\ref{e.comb1}). 

(ii) Substituting $\mu=0$ into (\ref{e.inskp22}) yields 
\[
s_k S_k^{-\frac{2\nu}{1+3\nu}} \geq \eps^{\frac{1-\nu}{1+3\nu}} 
\gamma_1^{-\frac{1+\nu}{1+3\nu}} L_\nu^{-\frac{2}{1+3\nu}}.
\]
It follows from $S_k \geq S_{k-1}$ and (\ref{e.inskp2}) that
\[
S_k^{\frac{1+\nu}{1+3\nu}}-S_{k-1}^{\frac{1+\nu}{1+3\nu}} \geq 
(S_k-S_{k-1})/\left(S_k^{1-\frac{1+\nu}{1+3\nu}}-S_{k-1}^{1-\frac{1+\nu}{1+3\nu}} \right) \geq \frac{1}{2} s_k S_k^{-\frac{2\nu}{1+3\nu}}
\geq 2^{-1} \eps^{\frac{1-\nu}{1+3\nu}} 
\gamma_1^{-\frac{1+\nu}{1+3\nu}} L_\nu^{-\frac{2}{1+3\nu}}.
\]
Let us sum up this inequality for $i=0,\dots,k$, giving
\[
S_k^{\frac{1+\nu}{1+3\nu}} \geq k 2^{-1} \eps^{\frac{1-\nu}{1+3\nu}} 
\gamma_1^{-\frac{1+\nu}{1+3\nu}} L_\nu^{-\frac{2}{1+3\nu}}.
\]
leading to
\[
S_k \geq \gamma_1^{-1} k^{\frac{1+3\nu}{1+\nu}} 2^{-\frac{1+3\nu}{1+\nu}}
\eps^{\frac{1-\nu}{1+\nu}} L_\nu^{-\frac{2}{1+\nu}}.
\]
This inequality and (\ref{e.comp}) give (\ref{e.comb11}). \qed
\end{proof}

The next corollary gives the complexity of ASGA-2 for attaining an 
$\eps$-solution of the problem (\ref{e.gfun}).

\begin{cor}\label{c.comp2}
Let $\{x_k\}_{k\geq 0}$ be generated by ASGA-2. Then \\
(i) If $\mu>0$, an $\eps$-solution of the problem (\ref{e.gfun}) is attained
by the complexity given in (\ref{e.compl1}) apart from some constants.\\
(ii) If $\mu=0$, an $\eps$-solution of the problem (\ref{e.gfun}) is attained
by the complexity given in (\ref{e.compl2}) apart from some constants.
\end{cor}

\begin{proof}
From $L_1 \geq \wt L$, $\alpha_0=1$, and the right hand side of (\ref{e.comb1}), we obtain
\[
2\left(\frac{1-\nu}{1+\nu}\right)^{\frac{1-\nu}{1+\nu}} 
L_{\nu}^{\frac{2}{1+\nu}} 
\left(1+\frac{\mu^{\frac{1+\nu}{1+3\nu}} \eps^{\frac{1-\nu}{1+3\nu}}}{2\gamma_1^{\frac{1+\nu}{1+3\nu}} L_\nu^{\frac{2}{1+3\nu}}} \right)^{-\frac{1+3\nu}{1+\nu}(k-1)} B_\omega(x^*,x_0) \leq \eps^{\frac{2}{1+\nu}}, 
\]
implying
\[
\begin{split}
\ln(A_2)-\ln(\eps^{\frac{2}{1+\nu}}) &\leq \frac{1+3\nu}{1+\nu}(k-1)
~\ln\left(1+2^{-1}\gamma_1^{-\frac{1+\nu}{1+3\nu}}\mu^{\frac{1+\nu}{1+3\nu}} \eps^{\frac{1-\nu}{1+3\nu}} L_\nu^{-\frac{2}{1+3\nu}}\right)\\
& \leq \frac{1+3\nu}{2(1+\nu)}~\gamma_1^{-\frac{1+\nu}{1+3\nu}} \mu^{\frac{1+\nu}{1+3\nu}} \eps^{\frac{1-\nu}{1+3\nu}} 
L_\nu^{-\frac{2}{1+3\nu}}(k-1),
\end{split}
\]
where
\[
A_2:= 2\left(\frac{1-\nu}{1+\nu}\right)^{\frac{1-\nu}{1+\nu}} 
L_{\nu}^{\frac{2}{1+\nu}} B_\omega(x^*,x_0).
\]
This leads to
\[
k \geq \frac{2(1+\nu)}{1+3\nu}~\gamma_1^{\frac{1+\nu}{1+3\nu}} \mu^{-\frac{1+\nu}{1+3\nu}} \eps^{-\frac{1-\nu}{1+3\nu}} 
L_\nu^{+\frac{2}{1+3\nu}} \left( \ln(A_2)+\ln(\eps^{-\frac{2}{1+\nu}}) \right),
\]
implying that (\ref{e.compl1}) is valid.

From (\ref{e.comb11}), we obtain
\[
\gamma_1 2^{\frac{1+3\nu}{1+\nu}} L_\nu^{\frac{2}{1+\nu}} 
k^{-\frac{1+3\nu}{1+\nu}} \eps^{-\frac{1-\nu}{1+\nu}} B_\omega(x^*,x_0) + \frac{\eps}{2} \leq \eps,
\]
leading to
\[
k \geq \gamma_1^{\frac{1+\nu}{1+3\nu}}2^{\frac{2+4\nu}{1+3\nu}} L_\nu^{\frac{2}{1+3\nu}} \eps^{-\frac{2}{1+3\nu}} B_\omega(x^*,x_0)^{\frac{1+\nu}{1+3\nu}},
\]
implying that (\ref{e.compl2}) is valid. \qed
\end{proof}

Theorems \ref{t.compl1} and \ref{t.compl2} provide the complexity of ASGA-1
and ASGA-2 for problems satisfying (\ref{e.holder}), where the same complexity
is attained apart from some constants.

\section{Double-subproblem accelerated (sub)gradient methods}
In this section we give two schemes for solving structured problems of
the form (\ref{e.gfun}) and investigate their complexity analysis, where the second one is a generalization of Nesterov's universal gradient method
\cite{NesU}. 

We generate a sequence of estimation functions $\{\phi_k(x)\}_{k \geq 0}$ that
approximate $h$ such that, for each iteration $k\geq 0$, 
\begin{equation}\label{e.ineqN}
S_k \left(h(y_k)-\frac{\eps}{2} \right) \leq \phi_k^*=\min_{x\in C} \phi_k(x),
\end{equation}
where $y_k \in V$ and $S_k$ is a scaling parameter given by (\ref{e.sk}). Let us consider the estimation sequence 
\begin{equation}\label{e.phikN}
\begin{array}{ll}
\phi_{k+1}(x):= \left\{
\begin{array}{ll}
B_\omega(x,x_0)                               
&~~ \mathrm{if}~ k=0,\vspace{3mm}\\
\phi_k(x)+ s_{k+1}\left[q_{k+1}(x,x_{k+1})+\psi(x)\right] 
&~~ \mathrm{if}~ k \in \mathbb{N},
\end{array}
\right. \vspace{3mm}\\
q_{k+1}(x,x_{k+1}):= \D f(x_{k+1})+\langle \nabla f(x_{k+1}), x-x_{k+1}\rangle
+\frac{\mu_f}{2}\|x-x_{k+1}\|^2.
\end{array}
\end{equation}
Let us define $\{v_k\}_{k\geq 0}$ as the sequence of minimizers of the estimation sequence 
$\{\phi_k\}_{k\geq 0}$, i.e.,
\begin{equation}\label{e.vk1}
v_{k+1} := \D \argmin_{x\in C}~ \phi_{k+1}(x). 
\end{equation}
The following result is necessary for providing the complexity of schemes
will be given in Section \ref{s.alg1}.

\begin{prop}\label{p.inphikN}
Let $\{\phi_k\}_{k\geq 0}$ be generated by (\ref{e.phikN}). Then (\ref{e.inphifk}) holds, and also if (\ref{e.ineqN}) is satisfied,
we have
\begin{equation}\label{e.compN}
h(y_k)-h(x^*) \leq \frac{B_\omega(x^*,x_0)}{S_k}+\frac{\epsilon}{2}.
\end{equation}
\end{prop}

\begin{proof}
The proof is given by induction on $k$. Since $S_0=0$ and 
$\phi_0(x)=B_\omega(x,x_0)$, the result is valid for $k=0$. We assume that is
true for $k$ and prove it for $k+1$. Then (\ref{e.phikN}) yields 
\[
\begin{split}
\phi_{k+1}(x)&= \phi_k(x)+ s_{k+1}\left(f(x_{k+1})+\langle \nabla f(x_{k+1}), x-x_{k+1} \rangle+\frac{\mu_f}{2}\|x-x_{k+1}\|^2+\psi(x)\right)\\
&\leq S_k h(x) + B_\omega(x,x_0)+ s_{k+1}\left(f(x_{k+1})+\langle \nabla f(x_{k+1}), x-x_{k+1} \rangle+\frac{\mu_f}{2}\|x-x_{k+1}\|^2 +\psi(x)\right)\\
&\leq S_k h(x) + B_\omega(x,x_0)+ s_{k+1} h(x) = S_{k+1} h(x) +
B_\omega(x,x_0).
\end{split}
\]

From (\ref{e.ineqN}) and (\ref{e.inphifk}), we obtain
\[
h(y_k) \leq \frac{\eps}{2}+\frac{1}{S_k} \phi_k^* \leq 
\frac{\eps}{2}+\frac{1}{S_k} \min_{x\in C} (S_k~ h(x) + B_\omega(x,x_0)) =
\frac{\eps}{2}+h(x^*)+\frac{B_\omega(x^*,x_0)}{S_k},
\]
implying (\ref{e.compN}) holds. \qed
\end{proof}

\subsection{{\bf Novel double-subproblem algorithms}} \label{s.alg1}
Here we give two new algorithms using the estimation sequence (\ref{e.phikN})
and investigate the related convergence analysis.

The following theorem shows that how the estimation sequence (\ref{e.phikN}) 
can be used to construct the sequence $\{x_k\}_{k\geq 0}$ guaranteeing
(\ref{e.ineqN}).

\begin{thm}\label{t.alg1N}
Let $f$ satisfies (\ref{e.holder}) with $L_{\nu} < +\infty$, 
$\alpha_k:=(s_{k+1}/S_{k+1})$ for $s_{k+1}>0$,
the sequence $\{v_k\}_{k\geq 0}$ be generated by (\ref{e.vk1}),
and
\begin{equation}\label{e.xk1N}
x_{k+1} := (1-\alpha_k)y_k+\alpha_k v_k.
\end{equation}
Let us also define
\begin{equation}\label{e.uk1N}
u_{k+1}:=\D \argmin_{x\in C} \left\{ B(x,v_k)+s_{k+1} \left(\langle \nabla f(x_{k+1}),x \rangle+\frac{\mu_f}{2}\|x-x_{k+1}\|^2+\psi(x)\right) \right\},
\end{equation}
\begin{equation}\label{e.yk1N}
y_{k+1} := (1-\alpha_k)y_k+\alpha_k u_{k+1}.
\end{equation}
We assume that (\ref{e.norm2}) holds for $y=y_{k+1}$, $z=x_{k+1}$, 
$\delta:=\eps \alpha_k$ with $\eps>0$, and (\ref{e.lhat1}) holds.
Then we have
\begin{equation}\label{e.inqksN}
\phi_{k+1}^* \geq S_{k+1} \left(h(y_{k+1})-\frac{\eps}{2}\right),
\end{equation}
if $s_{k+1}^2 \widehat{L}_{k+1}=(1+S_k\mu)S_{k+1}$.
\end{thm}

\begin{proof}
The proof is given by induction. Since $S_0=0$, the result for $k=0$ is
evident. We assume that (\ref{e.inqksN}) holds for some $k$ and show it for
$k+1$. 

Let us expand $\phi_k$ as
\lbeq{e.phik11N}
\phi_k(x)= 
B_\omega(x,x_0)+\sum_{i=1}^{k} s_i q_i(x,x_i)+S_k \psi(x).
\eeq 
Since $\psi$ is $\mu_p$-strongly convex, (\ref{e.phik11N}) implies that $\phi_k$ is $(1+S_k \mu)$-strongly convex. This and (\ref{e.phikN}) at $v_k$ yield  
\begin{equation}\label{e.eineN}
\phi_k(x)\geq \phi_k^*+\frac{1}{2}(1+S_k\mu) \|x-v_k\|^2~~~ \forall x\in C.
\end{equation}
From the induction assumption and the convexity of $f$, we obtain
\begin{equation}\label{e.inqks1N}
\phi_k^* \geq S_k\left(h(y_k)-\frac{\eps}{2} \right) \geq 
S_k\left(f(x_{k+1})+\langle \nabla f(x_{k+1}), y_k-x_{k+1} \rangle+\psi(y_k)-\frac{\eps}{2} \right).
\end{equation}
It follows from (\ref{e.xk1N}) that
\begin{equation}\label{e.simN}
\begin{split}
S_k(y_k-x_{k+1})+s_{k+1}(x-x_{k+1}) &= S_k y_k-S_{k+1}x_{k+1}+s_{k+1}x\\
&= S_k y_k-S_{k+1}(\alpha_k v_k+(1-\alpha_k)y_k)+s_{k+1}x
= s_{k+1}(x-v_k).
\end{split}
\end{equation}
Using this, (\ref{e.phikN}), (\ref{e.uk1N}), (\ref{e.eineN}), (\ref{e.inqks1N}), and (\ref{e.simN}), one can write
\begin{equation}\label{e.phik3N}
\begin{split}
\phi_{k+1}(x)&= \phi_k(x)+s_{k+1} \left(f(x_{k+1})+
\langle \nabla f(x_{k+1}), x-x_{k+1} \rangle 
+\frac{\mu_f}{2} \|x-x_{k+1}\|^2+\psi(x)\right)\\
&\geq \phi_k^*+\frac{1+S_k\mu}{2} \|x-v_k\|^2+ s_{k+1} \left[f(x_{k+1})+
\langle \nabla f(x_{k+1}), x-x_{k+1} \rangle +\psi(x)\right]\\
&\geq S_k\left(f(x_{k+1})+\langle \nabla f(x_{k+1}), y_k-x_{k+1} \rangle+\psi(y_k)-\frac{\eps}{2} \right)+\frac{1+S_k\mu}{2} \|x-v_k\|^2\\
&+s_{k+1}\left(f(x_{k+1})+ \langle \nabla f(x_{k+1}), x-x_{k+1} \rangle +\psi(x)\right)\\
&= S_{k+1} f(x_{k+1})+S_k \psi(y_k)+s_{k+1}\psi(x)-S_k\frac{\eps}{2}
+\frac{1+S_k\mu}{2} \|x-v_k\|^2 \\
&+\langle \nabla f(x_{k+1}), S_k(y_k-x_{k+1})+s_{k+1}(x-x_{k+1}) \rangle\\
&\geq S_{k+1} f(x_{k+1})+ s_{k+1}\langle \nabla f(x_{k+1}), x-v_k \rangle+S_k \psi(y_k)+s_{k+1}\psi(x)-S_k\frac{\eps}{2}
+\frac{1+S_k\mu}{2} \|x-v_k\|^2.
\end{split} 
\end{equation}
By the convexity of $\psi$ and (\ref{e.xk1N}), we get
\begin{equation}\label{e.psiinN}
S_k \psi(y_k)+s_{k+1}\psi(u_{k+1})=S_{k+1}(\alpha_k \psi(u_{k+1})+(1-\alpha_k)
\psi(y_k))\geq S_{k+1} \psi(y_{k+1}).
\end{equation}
The definition of $y_k$ and $x_{k+1}$ yield
\[
y_{k+1}-x_{k+1}= 
\alpha_k u_{k+1}+(1-\alpha_k)y_k-\alpha_k v_k-(1-\alpha_k)y_k
= \alpha_k (u_{k+1}-v_k).
\]
From this, (\ref{e.phik3N}), and (\ref{e.psiinN}), we obtain
\begin{equation}\label{e.phik4N}
\begin{split}
\phi_{k+1}^*&\geq S_{k+1} f(x_{k+1})+ s_{k+1}\langle \nabla f(x_{k+1}), u_{k+1}-v_k \rangle+S_k \psi(y_k)+s_{k+1}\psi(u_{k+1})\\
&-S_k\frac{\eps}{2}+\frac{1+S_k\mu}{2} \|u_{k+1}-v_k\|^2\\
&\geq S_{k+1} \left(f(x_{k+1})+\langle \nabla f(x_{k+1}), y_{k+1}-x_{k+1}\rangle +\psi(y_{k+1})\right)-S_k\frac{\eps}{2}+\frac{1+S_k\mu}{2\alpha_k^2} 
\|y_{k+1}-x_{k+1}\|^2.
\end{split} 
\end{equation}
By (\ref{e.norm2}) for $\delta=\alpha_k \eps$, we get
\[
f(x_{k+1})+\langle\nabla f(x_{k+1}),y_{k+1}-x_{k+1} \rangle \geq f(y_{k+1})
-\frac{\widehat{L}_{k+1}}{2} \|y_{k+1}-x_{k+1}\|^2-\frac{\alpha_k \eps}{2}.
\]
This and (\ref{e.phik4N}) give
\[
\begin{split}
\phi_{k+1}^* &\geq S_{k+1} \left(f(x_{k+1})+\langle \nabla f(x_{k+1}), y_{k+1}-x_{k+1}\rangle +\psi(y_{k+1})\right)-S_k\frac{\eps}{2}+\frac{1+S_k\mu}{2\alpha_k^2} 
\|y_{k+1}-x_{k+1}\|^2\\
&\geq S_{k+1} \left(f(y_{k+1})-\frac{\widehat{L}_{k+1}}{2} \|y_{k+1}-x_{k+1}\|^2-\frac{\alpha_k \eps}{2} \right)
+S_{k+1} \psi(y_{k+1})-S_k\frac{\eps}{2}+\frac{1+S_k\mu}{2\alpha_k^2}\|y_{k+1}-x_{k+1}\|^2\\
&= S_{k+1} \left(h(y_{k+1})-\frac{\eps}{2} \right)
+\frac{1}{2}\frac{S_{k+1}}{s_{k+1}^2}\left( (1+S_k\mu)S_{k+1}-s_{k+1}^2 \widehat{L}_{k+1} \right) \|y_{k+1}-x_{k+1}\|^2.
\end{split} 
\]
Therefore, setting $s_{k+1}^2 \widehat{L}_{k+1}=(1+S_k\mu)S_{k+1}$ implies
(\ref{e.inqksN}). \qed
\end{proof}

Theorem \ref{t.alg1N} leads to a simple scheme for solving problems of
the form (\ref{e.gfun}), which is summarized in the following. 

\vspace{4mm}
\begin{algorithm}[H] 
\DontPrintSemicolon 
\KwIn{ initial point $x_0 \in C$,~ $\nu$,~ $L_\nu$,~$\mu\geq 0$;~ $\eps>0$;}
\KwOut{$y_k$,~ $h_k$;}
\Begin{
    \While {stopping criteria do not hold}{
        compute $\widehat{L}_{k+1}$;~
        compute $s_{k+1}$ by (\ref{e.skp1});~$S_{k+1} = S_k+s_{k+1}$;~ 
        $\alpha_k = s_{k+1}/S_{k+1}$;\;
        $x_{k+1} = \alpha_k v_k+(1-\alpha_k) y_k$;~
        compute $u_{k+1}$ from (\ref{e.uk1N});~
        $y_{k+1}=\alpha_k u_{k+1}+(1-\alpha_k)y_k$;\; 
        compute $v_{k+1}$ from (\ref{e.vk1}); ~     
        $k=k+1$;\;
    }
    $h_k=h(y_k)$;
}
\caption{ {\bf ASGA-3} (double-subproblem ASGA)}
\label{a.aga} 
\end{algorithm}

\vspace{5mm}
ASGA-3 is a simple scheme which needs only two calls of the oracle per each
iteration. Therefore, we have that $N(k)=2k$ for ASGA-3. The same as ASGA-1, in ASGA-3 it is required to compute $\widehat{L}_{k+1}$ in each step. If the parameters $\nu$ and $L_\nu$ are available, then Proposition \ref{p.lhat} shows how to compute $\wh L_{k+1}$. Although ASGA-1 and ASGA-3 share some similarities, they have some basic differences: (i) they use different estimation sequences;
(ii) while ASGA-1 needs a single solution of (\ref{e.zk}), ASGA-3 requires one
solution of (\ref{e.vk1}) (Line 4) and a single solution of (\ref{e.uk1N}) (Line 5).

The subsequent two results give the complexity of ASGA-3. In view of Theorem \ref{t.alg1N}, the proofs are the same as Theorem \ref{t.compl1} and 
Corollary \ref{c.comp1}.

\begin{thm}\label{t.compl1N}
Let $\{y_k\}_{k\geq 0}$ be generated by ASGA-3. Then \\\\
(i) If $\mu>0$, we have
\[
h(y_k)-h(x^*) \leq \widehat{L}_1 \left(1+\frac{\mu^{\frac{1+\nu}{1+3\nu}}\eps^{\frac{1-\nu}{1+3\nu}}}{2 L_\nu^{\frac{2}{1+3\nu}}} 
\right)^{-\frac{1+3\nu}{1+\nu}(k-1)} B_\omega(x^*,x_0) + \frac{\eps}{2},
\]
where
\[
\widehat{L}_1=\left( \frac{1-\nu}{\eps(1+\nu)} \right)^{\frac{1-\nu}{1+\nu}}
L_{\nu}^{\frac{2}{1+\nu}}.
\]
(ii) If $\mu=0$, we have
\[
h(y_k)-h(x^*) \leq \left(\frac{2^{\frac{1+3\nu}{1+\nu}}L_{\nu}^{\frac{2}{1+\nu}}}{\eps^{\frac{1-\nu}{1+\nu}} k ^{\frac{1+3\nu}{1+\nu}}}
\right) B_\omega(x^*,x_0) + \frac{\eps}{2}.
\] 
\end{thm}

\begin{cor}\label{c.comp1N}
Let $\{y_k\}_{k\geq 0}$ be generated by ASGA-3. Then \\
(i) If $\mu>0$, an $\eps$-solution of the problem (\ref{e.gfun}) is attained
by the complexity given in (\ref{e.compl1}) apart from some constants.\\
(ii) If $\mu=0$, an $\eps$-solution of the problem (\ref{e.gfun}) is attained
by the complexity given in (\ref{e.compl2}) apart from some constants.
\end{cor}

In the following  we give a version ASGA-3 which does not need to know about
the parameters $\nu$ and $L_\nu$ using a backtracking line search guaranteeing
(\ref{e.norm2}). We describe the new scheme in the next result.

\begin{thm}\label{t.algN}
Let $f$ satisfies (\ref{e.holder}) with $L_{\nu} < +\infty$, 
$\alpha_k:=(s_{k+1}/S_{k+1})$ for $s_{k+1}>0$,
the sequences $\{v_k\}_{k\geq 0}$, $\{x_k\}_{k\geq 0}$, $\{u_k\}_{k\geq 0}$, 
and $\{y_k\}_{k\geq 0}$ be generated by (\ref{e.vk1}), (\ref{e.xk1N}),
(\ref{e.uk1N}), and (\ref{e.yk1N}), respectively, 
such that
\begin{equation}\label{e.blinN}
f(y_{k+1})\leq f(x_{k+1})+ \langle\nabla f(x_{k+1}),y_{k+1}-x_{k+1}\rangle+                \frac{L_{k+1}}{2} \|y_{k+1} - x_{k+1}\|^2+\frac{\alpha_k \eps}{2},
\end{equation} 
for $L_{k+1}\geq \widetilde{L}>0$. Then (\ref{e.inqksN}) is valid if 
$s_{k+1}^2 L_{k+1}=(1+S_k \mu)S_{k+1}$.
\end{thm}

\begin{proof}
Following the proof of Theorem \ref{t.alg1N}, the inequality (\ref{e.phik4N})
is valid. By (\ref{e.blinN}), for $\delta=\alpha_k \eps$, we get
\[
f(x_{k+1})+\langle\nabla f(x_{k+1}),y_{k+1}-x_{k+1} \rangle \geq f(y_{k+1})
-\frac{L_{k+1}}{2} \|y_{k+1}-x_{k+1}\|^2-\frac{\alpha_k\eps}{2}.
\]
By this and (\ref{e.phik4N}), we can write
\[
\begin{split}
\phi_{k+1}^*&\geq S_{k+1} \left(f(x_{k+1})+\langle \nabla f(x_{k+1}), y_{k+1}-x_{k+1} \rangle
+S_{k+1}\psi(y_{k+1})\right)-S_k\frac{\eps}{2}+\frac{1+S_k\mu}{2\alpha_k^2}
\|y_{k+1}-x_{k+1}\|^2\\
&\geq S_{k+1} \left(f(y_{k+1})-\frac{L_{k+1}}{2} \|y_{k+1}-x_{k+1}\|^2-
\frac{\alpha_k\eps}{2}  \right)
+S_{k+1} \psi(y_{k+1})-S_k\frac{\eps}{2}+\frac{1+S_k\mu}{2\alpha_k^2}
\|y_{k+1}-x_{k+1}\|^2\\
&= S_{k+1} \left(h(y_{k+1})-\frac{\eps}{2} \right)
+\frac{1}{2}\frac{S_{k+1}}{s_{k+1}^2}\left( (1+S_k\mu)S_{k+1}-s_{k+1}^2 L_{k+1} \right) \|y_{k+1}-x_{k+1}\|^2.
\end{split} 
\]
Therefore, setting $s_{k+1}^2 L_{k+1}=(1+S_k\mu)S_{k+1}$ yields that 
(\ref{e.inqks}) is valid. \qed
\end{proof}

In the light of Theorem \ref{t.algN} we give a simple iterative scheme for solving the problem (\ref{e.gfun}), where the sequences $\{v_k\}_{k\geq 0}$,
$\{x_k\}_{k\geq 0}$, $\{u_k\}_{k\geq 0}$, and $\{y_k\}_{k\geq 0}$ are generated
by (\ref{e.vk1}), (\ref{e.xk1N}), (\ref{e.uk1N}), and (\ref{e.yk1N}),
respectively. Proposition \ref{p.weld} shows that the condition (\ref{e.blinN})
holds if finite iterations of a backtracking line search is applied. We 
summarize the above-mentioned discussion in the subsequent algorithm:

\vspace{3mm}
\begin{algorithm}[H] 
\DontPrintSemicolon 
\KwIn{ initial point $x_0\in C$,~$L_0>0$,~$\gamma_1>1$,~$\gamma_2<1$, ~$p=0$,~$\mu\geq 0$;~$\eps>0$;}
\KwOut{$y_k$,~ $h_k$;}
\Begin{
    \While {stopping criteria do not hold}{
        \Repeat{$f(y_{k+1})\leq f(x_{k+1})+
                \langle\nabla f(x_{k+1}),y_{k+1}-x_{k+1}\rangle+
               \frac{1}{2}\overline{L}_{k+1} \|y_{k+1}-x_{k+1}\|^2
               +\frac{1}{2}\alpha_k\eps$}{
               $\overline{L}_{k+1} = \gamma_1^p L_k$;~ 
               compute $s_{k+1}$ by (\ref{e.skp2});~
               $\wh S_{k+1} = S_k+s_{k+1}$;~ 
               $\alpha_k = s_{k+1}/\wh S_{k+1}$;\;
               $\wh x_{k+1} = \alpha_k v_k+(1-\alpha_k) y_k$;
               compute $u_{k+1}$ by (\ref{e.uk1N}); 
               $\wh y_{k+1}=\alpha_k u_{k+1}+(1-\alpha_k)y_k$;
               $p=p+1$;      
        }        
        $x_{k+1}=\widehat{x}_{k+1}$; $y_{k+1}=\widehat{y}_{k+1}$;
        $u_{k+1}=\widehat{u}_{k+1}$; $S_{k+1}=\widehat{S}_{k+1}$;
        $L_{k+1}=\gamma_2\overline{L}_{k+1}$; \;
        compute $v_{k+1}$ by (\ref{e.vk1}); $k=k+1$;~$p=0$;\;
    }
    $h_k=h(y_k)$;
}
\caption{ {\bf ASGA-4} (parameter-free double-subproblem ASGA)}
\label{a.uga2} 
\end{algorithm}

\vspace{3mm}
The loop between Line 3 and Line 6 of ASGA-4 is called the inner cycle and the
loop between Line 2 and Line 9 of ASGA-4 is called the outer cycle. Hence
Proposition \ref{p.weld} shows that the inner cycle is terminated in a finite
number of iterations. ASGA-4 ans ASGA-2 share some similarities; however, they
use different estimation sequences; in each iteration ASGA-2 needs some
solutions of (\ref{e.zk}), while ASGA-4 requires a single solution of
(\ref{e.vk1}) (Line 8 in the outer cycle) and some solutions of (\ref{e.uk1N}) 
(Line 5 in the inner cycle). 

The following result gives the number of oracles $N(k)$ needed after $k$
iterations of ASGA-4.

\begin{prop}\label{p.tnocN}
Let $\{y_k\}_{k\geq 0}$ be generated by ASGA-4. Then 
\[
N(k)\leq 2\left( 1-\frac{\ln \gamma_2}{\ln \gamma_1} \right) (k+1)+
\frac{2}{\ln \gamma_1} \ln \frac{\gamma_1\gamma_2 \wt L}{L_0}.
\]
\end{prop}

Proposition \ref{p.tnoc} implies that ASGA-4 on average requires at most two
calls of the first-order oracle per iteration, while ASGA-3 needs exactly a single call of the first-order oracle per iteration. The proofs of the
following two results are the same as Theorem \ref{t.compl2} and 
Corollary \ref{c.comp2}.
 
\begin{thm}\label{t.compl2N}
Let $\{y_k\}_{k\geq 0}$ be generated by ASGA-4. Then \\\\
(i) If $\mu>0$, we have
\[
h(y_k)-h(x^*) \leq L_1 \left(1+\frac{\mu^{\frac{1+\nu}{1+3\nu}} \eps^{\frac{1-\nu}{1+3\nu}}}{2 \gamma_1^{\frac{1+\nu}{1+3\nu}} L_\nu^{\frac{2}{1+3\nu}}} 
\right)^{-\frac{1+3\nu}{1+\nu}(k-1)} B_\omega(x^*,x_0) + \frac{\eps}{2}, 
\]
where $L_1=2^{p_1}L_0$.\\
(ii) If $\mu=0$, we have
\[
h(y_k)-h(x^*) \leq \left(\frac{\gamma_1 2^{\frac{1+3\nu}{1+\nu}} L_{\nu}^{\frac{2}{1+\nu}}}{\eps^{\frac{1-\nu}{1+\nu}} k^{\frac{1+3\nu}{1+\nu}}}
\right) B_\omega(x^*,x_0) + \frac{\eps}{2}.
\] 
\end{thm}

\begin{cor}\label{c.comp2N}
Let $\{x_k\}_{k\geq 0}$ be generated by ASGA-4. Then \\
(i) If $\mu>0$, an $\eps$-solution of the problem (\ref{e.gfun}) is attained
by the complexity given in (\ref{e.compl1}) apart from some constants.\\
(ii) If $\mu=0$, an $\eps$-solution of the problem (\ref{e.gfun}) is attained
by the complexity given in (\ref{e.compl2}) apart from some constants.
\end{cor}

We here emphasis that the Nesterov-type optimal methods do not guarantee the convergence of the sequence of iteration points in general;
however, the next result shows that the sequence $\{x_k\}_{k\geq 0}$ generated by ASGA-1 or ASGA-2 (the sequence $\{y_k\}_{k\geq 0}$ generated by ASGA-3 or ASGA-2) is convergent to $x^*$ if the objective $h$ is strictly convex and 
$x^* \in \mathrm{int}~ C$, where $\mathrm{int}~ C$ denotes the interior of $C$.

\begin{prop} \label{p.strict}
Let $h$ be strictly convex. Then the sequence $\{x_k\}_{k\geq 0}$ generated by ASGA-1 or ASGA-2 is convergent to $x^*$ if $x^* \in \mathrm{int}~ C$.
\end{prop}

\begin{proof}
Strict convexity of $h$ implies that (\ref{e.gfun}) has the unique minimizer $x^*$. Since $x^* \in \mathrm{int}~ C$, there exists a small $\delta>0$ such that the convex
and compact neighborhood
\[
N(x^*) := \{x \in C \mid \|x-x^*\| \leq \delta\}
\]
is included in $C$. We set $x_\delta$ as a minimizer of the problem
\lbeq{e.min}
\begin{array}{ll}
\min          &~ h(x)\\
\mathrm{s.t.} &~ x \in \partial N(x^*),
\end{array}
\eeq
where $\partial N(x^*)$ denotes the boundary of $N(x^*)$. Let us define $\eps_\delta := h(x_\delta)-h^*$ and consider the upper level set
\[
N_h(x_\delta) := \{x \in C \mid h(x) \leq h(x_\delta) =h^*+\eps_\delta\}.
\]
For given $\eps_\delta$, Theorems \ref{t.compl1} and \ref{t.compl2} show that ASGA-1 and ASGA-2 attain an $\eps_\delta$-solution of (\ref{e.gfun}) in a finite number of iterations, say $\kappa$. Hence after $\kappa$ iterations the best
point $x_b$ satisfies $h(x_b) \leq h^*+\eps_\delta$, i.e., $x_b \in N_h(x_\delta)$. It remains to show $N_h(x_\delta) \subseteq N(x^*)$. By
contradiction, we suppose that there exists 
$\widetilde{x} \in N_h(x_\delta) \setminus N(x^*)$. Since 
$\widetilde{x} \not\in N(x^*)$, we have $\|\widetilde{x}-x^*\| > \delta$.
Therefore, there exists $\lambda_0 \in {]0,1[}$ such that
\[
\|\lambda_0 \widetilde{x} + (1-\lambda_0)x^*\| = \delta.
\]
From $\lambda_0 \widetilde{x} + (1-\lambda_0)x^* \in \partial N(x^*)$,
(\ref{e.min}), $h(\widetilde{x}) \leq h(x_\delta)$, and the strictly convex property of $h$, we obtain
\[
h(x_\delta) \leq h(\lambda_0 \widetilde{x} + (1-\lambda_0)x^*) < \lambda_0 
h(\widetilde{x}) + (1-\lambda_0) h(x^*) \leq \lambda_0 h(x_\delta) + (1-\lambda_0) h(x_\delta) = h(x_\delta),
\]
which is a contradiction, i.e., $N_h(x_\delta) \subseteq N(x^*)$ implying $x_b \in N(x^*)$ giving the results. \qed
\end{proof}

Note that the same proposition can be proved for ASGA-3 or ASGA-4 if we replace the sequence $\{x_k\}_{k\geq 0}$ by the sequence $\{y_k\}_{k\geq 0}$. It is
also valid for other Nesterov-type optimal methods.

\section{Applicability of accelerated (sub)gradient methods}
In this section we discuss some important aspects of efficient implementation 
of ASGA-1, ASGA-2, ASGA-3, and ASGA-4 for solving the problem (\ref{e.gfun}).

\subsection{{\bf Solving the auxiliary problems}}
To apply ASGA-1, ASGA-2, ASGA-3, and ASGA-4 to large problems of the form
(\ref{e.gfun}), we need to solve the auxiliary problems (\ref{e.zk}),   (\ref{e.vk1}), and (\ref{e.uk1N}) efficiently. In general, these problems
cannot be solved in a closed form; on the other hand, they can be handled
efficiently if $\psi$ and $C$ are simple enough and $\omega$ is selected
appropriately. In this section we show that they can be solved in a closed form
for several $\psi$ and $C$ appearing in applications. Let us emphasis that the
following results can be used in other Nesterov-type optimal methods either by
the same solution or by slightly modifications.

In the following two results we give a simplification of the auxiliary problem
(\ref{e.zk}) for the special case $\mu_f=0$ and $\psi\equiv 0$.

\begin{prop}\label{p.inphik}
Let $f$ be convex ($\mu_f=0$) and $\psi\equiv 0$ in (\ref{e.gfun}). Then the estimation sequence $\phi_k(z)$ (\ref{e.phik}) satisfies
\begin{equation}\label{e.ests1}
\phi_{k}(x) = \phi_{k}^* +B_\omega(x,z_{k}).
\end{equation}
Moreover, the  auxiliary problem (\ref{e.zk}) is simplified to
\begin{equation}\label{e.zk2}
z_{k+1} = \D \argmin_{x\in C} B_\omega(x,z_k)+s_{k+1}\langle \nabla f(y_k),x\rangle.
\end{equation}
\end{prop}

\begin{proof}
We first show (\ref{e.ests1}) by induction. For $k=0$, since $x_0\in C$, we have
\[
\phi_0^* = \min_{x\in C}~\phi_0(x) = \min_{x\in C}~B_\omega(x,z_0)=0,
\]
leading to $\phi_0(x)=\phi_0^*+B_\omega(x,z_0)$. We assume that is true for
$k-1$ and prove it for $k$. By substituting (\ref{e.ests1}) into (\ref{e.phik}), we get
\begin{equation}\label{e.phik1}
\phi_{k}(x)= \phi_{k-1}^* +B_\omega(x,z_{k-1})+ s_{k}(f(y_{k-1})+\langle \nabla f(y_{k-1}), x-y_{k-1} \rangle). 
\end{equation}
The first-order optimality condition of this identity gives
\[
\nabla B_\omega(\cdot,z_{k-1})(z_k)+s_{k} \nabla f(y_{k-1}) = 0,
\]
leading to 
\begin{equation}\label{e.dphik1}
\langle \nabla B_\omega(\cdot,z_{k-1})(z_k), x-z_k\rangle = 
- s_{k} \langle \nabla f(y_{k-1}), x-z_k \rangle. 
\end{equation}
Setting $x=z_{k}$ in (\ref{e.phik1}) yields
\begin{equation}\label{e.phik2}
\phi_{k}^*= \phi_{k-1}^* +B_\omega(z_{k},z_{k-1})+ s_{k}(f(y_{k-1})+
\langle \nabla f(y_{k-1}), z_{k}-y_{k-1} \rangle). 
\end{equation}
By subtracting (\ref{e.phik2}) from (\ref{e.phik1}), we get
\[
\phi_{k}(x)= \phi_{k}^*+B_\omega(x,z_{k-1})-B_\omega(z_{k},z_{k-1})+ 
s_{k} \langle \nabla f(y_{k-1}), x-z_{k} \rangle.
\]
From this and (\ref{e.dphik1}), we obtain
\[
\begin{split}
\phi_{k}(x)&= \phi_{k}^*+B_\omega(x,z_{k-1})-B_\omega(z_{k},z_{k-1})- 
\langle \nabla B_\omega(\cdot,z_{k-1})(z_k), x-z_{k}\rangle\\
&=\phi_{k}^*+\omega(x)-\omega(z_{k-1})-\langle \nabla \omega(z_{k-1}), x-z_{k-1} \rangle-\omega(z_{k})+\omega(z_{k-1})\\
&~~~+\langle \nabla \omega(z_{k-1}), z_{k}-z_{k-1} \rangle-\langle \nabla \omega(z_k)-\nabla\omega(z_{k-1}), x-z_{k}\rangle\\
&=\phi_{k}^*+\omega(x)-\omega(z_{k})-
\langle \nabla\omega(z_k), x-z_{k}\rangle
=\phi_{k}^*+B_\omega(x,z_k),
\end{split}
\]
giving (\ref{e.ests1}). 

It follows from (\ref{e.zk2}) and (\ref{e.ests1}) that
\[
\begin{split}
z_{k+1} &= \argmin_{x\in C} \phi_{k+1}(x)=  
\argmin_{x\in C} \phi_k^* +B_\omega(x,z_k)+ s_{k+1}(f(y_k)+
\langle \nabla f(y_k), x-y_k \rangle)\\
&=\argmin_{x\in C} B_\omega(x,z_k)+s_{k+1}\langle \nabla f(y_k),x\rangle,
\end{split}
\]
giving the result. \qed
\end{proof}

Let us consider the prox-function 
\begin{equation}\label{e.proxq}
\omega(x) := \frac{1}{2} \|x-x_0\|_2^2.
\end{equation}
From the definition of the Bregman distance $B_\omega(x,y)$, we obtain
\[
B_\omega(x,y) = \frac{1}{2} \|x-x_0\|_2^2+\frac{1}{2} \|y-x_0\|_2^2-
\langle x-x_0,y-x+x-x_0 \rangle =\frac{1}{2} \|x-y\|_2^2,
\]
which is the Euclidean distance. We note that using (\ref{e.phik11}) the auxiliary problem (\ref{e.zk}) with $\omega$ defined by (\ref{e.proxq}) is
strongly convex and then has an unique solution. 

We are in a position to give the solution of (\ref{e.zk2}) for convex
problems with $\psi\equiv 0$.

\begin{prop}\label{p.proj}
Let $f$ be convex ($\mu_f=0$) and $\psi\equiv 0$ in (\ref{e.gfun}). Then the global minimizer $z_{k+1}$ of
(\ref{e.zk2}) satisfies
\begin{equation}\label{e.optsub1}
\nabla \omega(z_k)-\nabla \omega(z_{k+1})-s_{k+1} \nabla f(y_k) \in s_{k+1}
\partial \psi(z_{k+1}) + N_C(z_{k+1}).
\end{equation}
Moreover, if $\omega$ is given by (\ref{e.proxq}), the solution $z_{k+1}$ of
(\ref{e.zk2}) is given by
\begin{equation}\label{e.proj}
z_{k+1} = P_C(z_k-s_{k+1} \nabla f(y_k)).
\end{equation}
\end{prop}

\begin{proof}
From (\ref{e.fopt}) for (\ref{e.zk2}), we obtain
\[
\begin{split}
0 &\in \nabla B_\omega(\cdot,z_k)(z_{k+1})+s_{k+1} \nabla f(y_k)+s_{k+1}
\partial \psi(z_{k+1}) + N_C(z_{k+1})\\
&=\nabla\omega(z_{k+1})-\nabla\omega(z_k)+s_{k+1} \nabla f(y_k)+s_{k+1}
\partial \psi(z_{k+1}) + N_C(z_{k+1}),
\end{split}
\] 
implying (\ref{e.optsub1}) is valid.

By $\psi\equiv 0$ and (\ref{e.opts1}), we get
\[
z_k-z_{k+1}-s_{k+1} \nabla f(y_k) \in N_C(z_{k+1}).
\]
This is the optimality condition of the problem
\[
\D \min_{z\in C}~ \frac{1}{2} \|z-(z_k-s_{k+1}\nabla f(y_k))\|_2^2,
\]
which is the orthogonal projection of $z_k-s_{k+1}\nabla f(y_k)$ onto $C$,
giving the result.\qed
\end{proof}

For $\mu_f=0$, $\psi\equiv 0$, and $\omega$ given by (\ref{e.proxq}), Proposition \ref{p.proj} implies that the auxiliary problem (\ref{e.zk}) can 
be solved efficiently if the orthogonal projection onto the convex domain $C$
is cheaply available. There are many important convex domains that the
orthogonal projection onto them is efficiently available either in a closed
form or by a simple iterative scheme (see Table 5.1 in \cite{AhoT}).

The auxiliary problems (\ref{e.zk}) and (\ref{e.vk1}) have the same structure;
in contrast, (\ref{e.zk}) involves $\{y_k\}_{k\geq 0}$ while (\ref{e.vk1})
includes $\{x_k\}_{k\geq 0}$. Therefore, we only consider (\ref{e.vk1}) in the
remainder of this section. The next result gives optimality conditions for
(\ref{e.vk1}) and (\ref{e.uk1N}).

\begin{prop}\label{p.phi0}
Let $v_{k+1}$ and $u_{k+1}$ be the global minimizer of (\ref{e.vk1}) and
(\ref{e.uk1N}), respectively. Then
\begin{equation}\label{e.optsub}
\nabla\omega(x_0)-\nabla\omega(v_{k+1})-\sum_{i=1}^{k+1} s_i (\nabla f(x_i) +\mu_f (v_{k+1}-x_i))\in S_{k+1} \partial\psi(v_{k+1})+N_C(v_{k+1}),
\end{equation}
\begin{equation}\label{e.optsubn}
\nabla\omega(v_k)-\nabla\omega(u_{k+1})- s_{k+1} (\nabla f(x_{k+1}) +\mu_f (u_{k+1}-x_{k+1}))\in s_{k+1} \partial\psi(z_{k+1})+ N_C(u_{k+1}).
\end{equation}
\end{prop}

\begin{proof}
From (\ref{e.breg}), we obtain $\nabla B_\omega(\cdot,x_0)(x) = \nabla\omega(x)- \nabla\omega(x_0)$. By this, (\ref{e.fopt}), and
(\ref{e.phik11N}) for the auxiliary problem (\ref{e.vk1}), we get
\[
\begin{split}
0 &\in \nabla B_\omega(\cdot,x_0)(v_{k+1})+\sum_{i=1}^{k+1} s_i (\nabla f(x_i)+\mu_f (v_{k+1}-x_i))+S_{k+1}\partial\psi(v_{k+1})
+N_C(v_{k+1})\\
&=\nabla\omega(v_{k+1})-\nabla\omega(x_0)+\sum_{i=1}^{k+1} s_i (\nabla f(x_i) +\mu_f (v_{k+1}-x_i))+S_{k+1} \partial\psi(v_{k+1})+ N_C(v_{k+1}),
\end{split}
\] 
implying (\ref{e.optsub}) is valid. 

It follows from (\ref{e.fopt}) and (\ref{e.uk1N}) that
\[
\begin{split}
0 &\in \nabla\omega(u_{k+1})-\nabla\omega(v_k)+s_{k+1} (\nabla f(x_{k+1}) +\mu_f (u_{k+1}-x_{k+1}))+s_{k+1} \partial\psi(u_{k+1})+ N_C(u_{k+1}),
\end{split}
\] 
giving (\ref{e.optsubn}). \qed
\end{proof}

We now consider a simple case of (\ref{e.gfun}) with $C=V$. We verify
the solution of the auxiliary problems (\ref{e.vk1}) and (\ref{e.uk1N}) in the following result.

\begin{prop}\label{p.prox}
Let $C=V$ and $\omega$ be given by (\ref{e.proxq}). Then the solution $v_{k+1}$
of the auxiliary problem (\ref{e.vk1}) is given by
\begin{equation}\label{e.prox}
v_{k+1} = \mathrm{prox}_{\wt\lambda \psi}(\wt y),
\end{equation}
where 
\begin{equation}\label{e.wty}
\wt\lambda:=\frac{S_{k+1}}{1+\mu_f S_{k+1}},~~~\wt y:= \frac{1}{1+\mu_f S_{k+1}} \left(x_0-\sum_{i=1}^{k+1} s_i \left(\nabla f(x_i)-\mu_f x_i\right)\right).
\end{equation}
Moreover, the solution $u_{k+1}$ of the auxiliary problem (\ref{e.uk1N}) is given by
\begin{equation}\label{e.prox1}
u_{k+1} = \mathrm{prox}_{\wh\lambda \psi}(\wh y),
\end{equation}
where 
\begin{equation}\label{e.wty1}
\wh\lambda:=\frac{s_{k+1}}{1+\mu_f s_{k+1}},~~~\wh y:= \frac{1}{1+\mu_f s_{k+1}} \left(v_k- s_{k+1} \left(\nabla f(x_{k+1})-\mu_f x_{k+1}\right)\right).
\end{equation}
\end{prop}

\begin{proof}
By (\ref{e.optsub}) for (\ref{e.vk1}) with $C=V$, we get
\[
x_0-v_{k+1}-\sum_{i=1}^{k+1} s_i (\nabla f(x_i) +\mu_f (v_{k+1}-x_i)) \in S_{k+1} \partial\psi(v_{k+1}),
\]
or equivalently 
\[
0 \in (1+\mu_f S_{k+1})v_{k+1}-x_0+\sum_{i=1}^{k+1} s_i (\nabla f(x_i) +\mu_f (v_{k+1}-x_i))+S_{k+1}  \partial\psi(v_{k+1}).
\]
This is the optimality condition (\ref{e.opts}) of the problem
\[
\D\min_{z \in V} ~\frac{1}{2} \|z-\wt y\|_2^2+S_{k+1}(1+\mu_f S_{k+1})^{-1} \psi(z),
\]
where $z_{k+1}$ is the unique minimizer of this problem with $\wt \lambda$ and
$\wt y$ given by (\ref{e.wty}). 

By (\ref{e.optsub}) for (\ref{e.uk1N}) with $C=V$, we get
\[
0 \in (1+\mu_f s_{k+1})u_{k+1}-\left(v_k-s_{k+1} (\nabla f(x_{k+1})+\mu_f x_{k+1})\right)+s_{k+1} \partial\psi(u_{k+1}).
\]
This is the optimality condition (\ref{e.opts}) of the problem
\[
\D\min_{z \in V} ~\frac{1}{2} \|z-\wh y\|_2^2+s_{k+1}(1+\mu_f s_{k+1})^{-1} \psi(z),
\]
where $z_{k+1}$ is the unique minimizer of this problem with $\wh \lambda$ and
$\wh y$ given by (\ref{e.wty1}).\qed
\end{proof}

Proposition \ref{p.prox} implies that if $C=V$, then the auxiliary problems
(\ref{e.vk1}) and (\ref{e.uk1N}) are reduced to proximal problems which is a
well-studied subject in convex optimization. More precisely, the proximal
problems (\ref{e.prox}) and (\ref{e.prox1}) can be solved for many simple
convex functions $\psi$ appearing in applications either in a closed form or 
by a simple iterative scheme (see, e.g., Table 6.1 in \cite{AhoT}).

In the reminder of this section we consider cases that both $\psi$ and $C$ are
simple enough such that the auxiliary problems (\ref{e.vk1}) and (\ref{e.uk1N})
can be solved in a closed form. In particular we discuss the box constraints 
$C=\{x\in\mathbb{R}^n \mid x\in \x=[\underline{x},\overline{x}]\}$.

\begin{prop}\label{p.bcopt}
Let $C=\{x\in\mathbb{R}^n \mid x\in \x=[\underline{x},\overline{x}]\}$ and $\omega$ be given by (\ref{e.proxq}). There exists 
$g\in \partial \psi(v_{k+1})$ such that the solution $v_{k+1}$ of the auxiliary problem (\ref{e.vk1}) satisfies
\begin{equation}\label{e.bcopt1}
\forall j = 1, \dots, n,~~~v_{k+1}^j=\left\{
\begin{array}{ll}
\underline{x}^j & ~~ \mathrm{if} ~  
(1+\mu_f S_{k+1}) \underline{x}^j-x_0^j+\sum_{i=1}^{k+1} s_i \left(\nabla f(x_i)-\mu_f x_i\right)^j+S_{k+1} g^j\geq 0,\\
\overline{x}^j  & ~~ \mathrm{if} ~ 
(1+\mu_f S_{k+1}) \overline{x}^j-x_0^j+\sum_{i=1}^{k+1} s_i \left(\nabla f(x_i)-\mu_f x_i\right)^j+S_{k+1} g^j\leq 0,\\
t_1^j & ~~ \mathrm{if} ~
\underline{x}^j < t_1^j < \overline{x}^j,
\end{array}
\right.
\end{equation}
where
\[
t_1:=\D \frac{1}{1+\mu_f S_{k+1}} \left( x_0^j-\sum_{i=1}^{k+1}s_i (\nabla f(x_i)^j-\mu_f x_i)-S_{k+1}g^j \right).
\]
There exists $g\in \partial \psi(u_{k+1})$ such that the solution $u_{k+1}$ of the auxiliary problem (\ref{e.uk1N}) satisfies
\begin{equation}\label{e.bcopt1n}
\forall j = 1, \dots, n,~~~u_{k+1}^j=\left\{
\begin{array}{ll}
\underline{x}^j & ~~ \mathrm{if} ~ 
(1+\mu_f s_{k+1}) \underline{x}^j-v_k^j+s_{k+1} \left(\nabla f(x_{k+1})-\mu_f x_{k+1}\right)^j+s_{k+1} g^j \geq 0,\\
\overline{x}^j  & ~~ \mathrm{if} ~  
(1+\mu_f s_{k+1}) \overline{x}^j-v_k^j+s_{k+1} \left(\nabla f(x_{k+1})-\mu_f x_{k+1}\right)^j+s_{k+1} g^j\leq 0,\\
t_2^j & ~~ \mathrm{if} ~
\underline{x}^j < t_2^j < \overline{x}^j,
\end{array}
\right.
\end{equation}
where
\[
t_2:=\D \frac{1}{1+\mu_f s_{k+1}} \left( v_k^j-s_{k+1} (\nabla f(x_{k+1})^j-\mu_f x_{k+1})-s_{k+1}g^j \right).
\]
\end{prop}

\begin{proof}
From (\ref{e.optsub}) and the definition of $N_\x(v_{k+1})$, there exists 
$g \in \partial\psi(v_{k+1})$ such that 
\begin{equation}\label{e.optc2}
0 \in \left\{(1+\mu_f S_{k+1}) v_{k+1}-x_0+\sum_{i=1}^{k+1} s_i \left(\nabla f(x_i)-\mu_f x_i\right)+S_{k+1} g+q ~\Big |~ \langle q,v_{k+1}-z\rangle\geq 0~~ \forall z \in \x \right\}.
\end{equation}
Deriving the $j$th component of $v_{k+1}$ involves three possibilities: (i) $v_{k+1}^j = \underline{x}^j$; (ii) $v_{k+1}^j = \overline{x}^j$; (iii) $\underline{x}^j < v_{k+1}^j < \overline{x}^j$. In Case (i), 
$v_{k+1}^j-z^j\leq 0$ for all $z \in \x$ implying $q^j \leq 0$. Then (\ref{e.optc2}) implies that
\[
(1+\mu_f S_{k+1}) v_{k+1}^j-x_0^j+\sum_{i=1}^{k+1} s_i \left(\nabla f(x_i)-\mu_f x_i\right)^j+S_{k+1} g^j \geq 0,
\]
for $v_{k+1}^j = \underline{x}^j$. In Case (ii), $v_{k+1}^j - z^j \geq 0$ for all $z \in \x$ so that $q^j \geq 0$. Hence (\ref{e.optc2}) yields  
\[
(1+\mu_f S_{k+1}) v_{k+1}^j-x_0^j+\sum_{i=1}^{k+1} s_i \left(\nabla f(x_i)-
\mu_f x_i\right)^j+S_{k+1} g^j \leq 0,
\]
for $v_{k+1}^j = \overline{x}^j$. In Case (iii), we have $v_{k+1}^j - z^j \geq 0$ for some $z \in \x$ and 
$v_{k+1}^j- z^j \leq 0$ for some other $z \in \x$. This leads to $q^j = 0$ implying 
\[
(1+\mu_f S_{k+1}) v_{k+1}^j-x_0^j+\sum_{i=1}^{k+1} s_i \left(\nabla f(x_i)-\mu_f x_i\right)^j+S_{k+1} g^j = 0.
\] 
These three cases lead to
\[
(1+\mu_f S_{k+1}) v_{k+1}^j-x_0^j+\sum_{i=1}^{k+1} s_i \left(\nabla f(x_i)-\mu_f x_i\right)^j+S_{k+1} g^j \left\{
\begin{array}{ll}
\geq 0 & ~~ \mathrm{if} ~ v_{k+1}^j = \underline{x}^j,\\
\leq 0 & ~~ \mathrm{if} ~ v_{k+1}^j = \overline{x}^j,\\
= 0    & ~~ \mathrm{if} ~ \underline{x}^j < v_{k+1}^j < \overline{x}^j.
\end{array}
\right.
\]
Computing $v_{k+1}$ from this equation implies (\ref{e.bcopt1}).

By (\ref{e.optsubn}) and the definition of $N_\x(u_{k+1})$, there exists 
$g \in \partial\psi(u_{k+1})$ such that 
\begin{equation}\label{e.optc22}
0 \in \left\{(1+\mu_f s_{k+1}) u_{k+1}-v_k+ s_{k+1} \left(\nabla f(x_{k+1})- \mu_f x_{k+1}\right)+s_{k+1} g+q \mid \langle q,u_{k+1}-z\rangle\geq 0~~ \forall z \in \x \right\}.
\end{equation}
To compute $u_{k+1}^j$ we consider three possibilities: 
(i) $u_{k+1}^j = \underline{x}^j$; (ii) $u_{k+1}^j = \overline{x}^j$; (iii) $\underline{x}^j < u_{k+1}^j < \overline{x}^j$. In Case (i), 
$u_{k+1}^j-z^j\leq 0$ for all $z \in \x$ implying $q^j \leq 0$. Then (\ref{e.optc22}) leads to
\[
(1+\mu_f s_{k+1}) u_{k+1}^j-v_k^j+s_{k+1} \left(\nabla f(x_{k+1})-
\mu_f x_{k+1}\right)^j+s_{k+1} g^j \geq 0,
\]
for $v_{k+1}^j = \underline{x}^j$. In Case (ii), $v_{k+1}^j - z^j \geq 0$ for all $z \in \x$ so that $q^j \geq 0$. This, together with (\ref{e.optc22}), implies  
\[
(1+\mu_f s_{k+1}) u_{k+1}^j-v_k^j+s_{k+1} \left(\nabla f(x_{k+1})-
\mu_f x_{k+1}\right)^j+s_{k+1} g^j \leq 0,
\]
for $v_{k+1}^j = \overline{x}^j$. In Case (iii), we have 
$u_{k+1}^j - z^j \geq 0$ for some $z \in \x$ and $u_{k+1}^j- z^j \leq 0$ for some other $z \in \x$, i.e., $q^j=0$ leading to 
\[
(1+\mu_f s_{k+1}) u_{k+1}^j-x_0^j+s_{k+1} \left(\nabla f(x_{k+1})-
\mu_f x_{k+1}\right)^j+s_{k+1} g^j = 0.
\] 
These three cases leads to
\[
(1+\mu_f s_{k+1}) u_{k+1}^j-v_k^j+s_{k+1} \left(\nabla f(x_{k+1})-
\mu_f x_{k+1}\right)^j+s_{k+1} g^j \left\{
\begin{array}{ll}
\geq 0 & ~~ \mathrm{if} ~ u_{k+1}^j = \underline{x}^j,\\
\leq 0 & ~~ \mathrm{if} ~ u_{k+1}^j = \overline{x}^j,\\
= 0    & ~~ \mathrm{if} ~ \underline{x}^j < u_{k+1}^j < \overline{x}^j,
\end{array}
\right.
\]
giving (\ref{e.bcopt1n}).\qed
\end{proof}

To show the applicability of Proposition \ref{p.bcopt}, we consider a special
case $\psi(\cdot) = \|\cdot\|_1$, which has been widely used in the fields
of sparse optimization and compressed sensing, see, e.g., \cite{Can,Don}.
We first need the following proposition. We use this result in 
Section \ref{s.elas}.

\begin{prop} \label{p.subd}
\cite[Proposition 2.3]{AhoT} Let $\phi:V \rightarrow \mathbb{R},~ \phi(x) = \|x\|$. Then 
\[
\partial \phi (x)= \left\{
\begin{array}{ll}
\{g \in V^* \mid \|g\|_* \leq 1\} & ~~ \mathrm{if} ~ x=0,\vspace{2mm}\\
\{g \in V^* \mid \|g\|_* = 1,~ \langle g,x \rangle = \|x\| \} & ~~ \mathrm{if} ~ x \neq 0.
\end{array}
\right.
\]
\end{prop}

\begin{prop} \label{p.soll1}
Let $C=\{x\in\mathbb{R}^n \mid x\in \x=[\underline{x},\overline{x}]\}$ and $\omega$ be given by (\ref{e.proxq}). Let also
\lbeq{e.ind}
\D \kappa(\wh p):=\sum_{\wh p^i<0} \wh p^i \underline{x}+
\sum_{\wh p^i>0} \wh p^i \overline{x},
\eeq
where $\wh p=x_0-\sum_{i=1}^{k+1}s_i(\nabla f(x_i)+\mu_f x_i)-S_{k+1} {\bf 1}$ (${\bf 1}$ is the vector of all ones) for (\ref{e.vk1}) and 
$\wh p=v_k-s_{k+1}(\nabla f(x_{k+1})+\mu_f x_{k+1})-s_{k+1} {\bf 1}$ for
(\ref{e.uk1N}). Then the global minimizer of the auxiliary problem (\ref{e.vk1}) for $\psi(x) = \|x\|_1$ is given by
\begin{equation}\label{e.soll11}
\forall j = 1, \dots, n,~~~ v_{k+1}^j = \left\{
\begin{array}{ll}
\underline{x}^j           & ~~ \mathrm{if} ~ \kappa(\wt q)>0,~c_1^j \geq 0,\\
\overline{x}^j            & ~~ \mathrm{if} ~\kappa(\wt q)>0,~ c_2^j \leq 0,\\
c_3^j & ~~ \mathrm{if} ~\kappa(\wt q)>0,~ c_3^j>0,\\
c_4^j & ~~ \mathrm{if} ~\kappa(\wt q)>0,~ c_4^j<0,\\
0& ~~ \mathrm{otherwise},
\end{array}
\right.
\end{equation}
where
\[
\begin{array}{l}
c_1:=\D(1+\mu_f S_{k+1})\underline{x}-x_0+\sum_{i=1}^{k+1} s_i (\nabla f(x_i)+\mu_f x_i)+S_{k+1}\mathrm{sign} (\underline{x}),\\
c_2:=\D(1+\mu_f S_{k+1})\overline{x}-x_0+\sum_{i=1}^{k+1} s_i (\nabla f(x_i)+\mu_f x_i)+S_{k+1}\mathrm{sign} (\overline{x}),\\
c_3:=\D \frac{1}{1+\mu_f S_{k+1}} \left(x_0-\sum_{i=1}^{k+1} s_i (\nabla f(x_i)+\mu_f x_i)-S_{k+1} {\bf 1}\right),\\
c_4:=\D\frac{1}{1+\mu_f S_{k+1}} \left(x_0-\sum_{i=1}^{k+1} s_i (\nabla f(x_i)+\mu_f x_i)+S_{k+1} {\bf 1}\right).
\end{array} 
\]
The global minimizer of the auxiliary problem (\ref{e.uk1N}) for 
$\psi(x) = \|x\|_1$ is given by
\begin{equation}\label{e.soll11}
\forall j = 1, \dots, n,~~~ u_{k+1}^j = \left\{
\begin{array}{ll}
\underline{x}^j           & ~~ \mathrm{if} ~ \kappa(\wt q)>0,~c_5^j \geq 0,\\
\overline{x}^j            & ~~ \mathrm{if} ~\kappa(\wt q)>0,~ c_6^j \leq 0,\\
c_7^j & ~~ \mathrm{if} ~\kappa(\wt q)>0,~ c_7^j>0,\\
c_8^j & ~~ \mathrm{if} ~\kappa(\wt q)>0,~ c_8^j<0,\\
0& ~~ \mathrm{otherwise},
\end{array}
\right.
\end{equation}
where
\[
\begin{array}{l}
c_5:=\D(1+\mu_f s_{k+1})\underline{x}-v_k+ s_{k+1} (\nabla f(x_{k+1})+\mu_f x_{k+1})+s_{k+1}\mathrm{sign} (\underline{x}),\\
c_6:=\D(1+\mu_f s_{k+1})\overline{x}-v_k+ s_{k+1} (\nabla f(x_{k+1})+\mu_f x_{k+1})+s_{k+1}\mathrm{sign} (\overline{x}),\\
c_7:=\D \frac{1}{1+\mu_f s_{k+1}} \left(v_k-s_{k+1} (\nabla f(x_{k+1})+\mu_f x_{k+1})-s_{k+1} {\bf 1}\right),\\
c_8:=\D\frac{1}{1+\mu_f s_{k+1}} \left(v_k-s_{k+1} (\nabla f(x_{k+1})+\mu_f x_{k+1})+s_{k+1} {\bf 1}\right).
\end{array}
\]
\end{prop}

\begin{proof}
Proposition \ref{p.subd} for $\psi(x) = \|x\|_1$ leads to
\begin{equation}\label{e.subdl1}
\partial \|x\|_1= \left\{
\begin{array}{ll}
\{g \in \mathbb{R}^n \mid \|g\|_\infty \leq 1\} & ~~ \mathrm{if} ~ x=0,\vspace{2mm}\\
\{g \in \mathbb{R}^n \mid \|g\|_\infty = 1,~ \langle g,x \rangle = \|x\|_1 \} & ~~ \mathrm{if} ~ x \neq 0.
\end{array}
\right.
\end{equation}
We first show $v_{k+1}=0$ if and only if $\kappa(\wh p)\leq 0$. By the
definition of the normal cone of $\x$ at 0, we have
\[
N_{\x}(0) = \left \{ p \in V \mid \forall z \in [\underline{x}, \overline{x}], \langle p, z \rangle \leq 0 \right\} = \left \{ p \in V ~\Big |~  \sum_{p_i<0} p_i \underline{x}+\sum_{p_i>0}p_i \overline{x} \leq 0 \right\}.
\]
From (\ref{e.optc2}), $z_{k+1}=0$ if and only if there exists 
\[
p\in N_{\x}(0)\bigcap \left(x_0-\sum_{i=1}^{k+1} s_i (\nabla f(x_i)+\mu_f x_i)-S_{k+1} \partial\psi(0)\right).
\]
By (\ref{e.subdl1}), this is possible if and only if 
\[
\min \left\{ \sum_{p_i<0} p_i \underline{x}+\sum_{p_i>0}p_i \overline{x}~\Big |~ p=x_0-\sum_{i=1}^{k+1} s_i (\nabla f(x_i)+\mu_f x_i)-S_{k+1} g, 
~\|g\|_\infty \leq 1 \right\} \leq 0.
\]
The solution of this problem is 
$\wh p=x_0-\sum_{i=1}^{k+1} s_i (\nabla f(x_i)+\mu_f x_i)-S_{k+1}{\bf 1}$.
Hence the minimum of this problem is given by (\ref{e.ind}). This implies $v_{k+1}=0$ if and only if
$\kappa(\wh p)\leq 0$. 

Let us assume $v_{k+1}\neq 0$, i.e., $\kappa(\wt q) > 0$. From (\ref{e.subdl1}), we obtain
\[
\partial \|v_{k+1} \|_1=\{g \in \mathbb{R}^n \mid \|g\|_\infty = 1,~ \langle g,v_{k+1} \rangle = \|v_{k+1}\|_1 \},
\]
leading to
\[
\sum_{j=1}^n (g^j v_{k+1}^j - |v_{k+1}^j|) = 0. 
\]
By induction on nonzero elements of $v_{k+1}$, we get $g^i v_{k+1}^i = |v_{k+1}^i|$, for $i = 1, \dots, n$. This implies that $g^i = \mathrm{sign} (\wh z^i)$ if $v_{k+1}^i \neq 0$. This implies 
\[
(1+\mu_f S_{k+1}) v_{k+1}^j-x_0^j+\sum_{i=1}^{k+1} s_i (\nabla f(x_i)+\mu_f x_i)+S_{k+1}(\partial \|v_{k+1}\|_1)^j \left\{
\begin{array}{ll}
\geq 0 &~~ \mathrm{if}~ v_{k+1}^j = \underline{x}^j,\\
\leq 0 &~~ \mathrm{if}~ v_{k+1}^j = \overline{x}^j,\\
= 0    &~~ \mathrm{if}~ \underline{x}^j < v_{k+1}^j < \overline{x}^j,
\end{array}
\right.
\]
for $j = 1, \dots, n$, leading to
\lbeq{e.optcl1}
(1+\mu_f S_{k+1})v_{k+1}^j-x_0^j+\sum_{i=1}^{k+1} s_i (\nabla f(x_i)+\mu_f x_i)+S_{k+1}\mathrm{sign} (v_{k+1}^j) \left\{
\begin{array}{ll}
\geq 0 &~~ \mathrm{if}~ v_{k+1}^j = \underline{x}^j,\\
\leq 0 &~~ \mathrm{if}~ v_{k+1}^j = \overline{x}^j,\\
= 0    &~~ \mathrm{if}~ \underline{x}^j < v_{k+1}^j < \overline{x}^j.
\end{array}
\right.
\eeq
Substituting $v_{k+1}^j = \underline{x}^j$ in (\ref{e.optcl1}) implies 
$c_1^j \geq 0$. If $v_{k+1}^j = \overline{x}^j$, we have $c_1^j \leq 0$.
If $\underline{x}^j < v_{k+1}^j < \overline{x}^j$, there are three possibilities: (i) $v_{k+1}^j > 0$; (ii) $v_{k+1}^j < 0$; (iii) 
$v_{k+1}^j = 0$. In Case (i), $\mathrm{sign} (v_{k+1}^j) = 1$ and
(\ref{e.optcl1}) lead to $v_{k+1}^j =c_2^j>0$. In Case (ii),  
$\mathrm{sign} (\wh z^i) = -1$ and (\ref{e.optcl1}) imply 
$v_{k+1}^j =c_3^j<0$. In Case (c), we get $v_{k+1}^j = 0$.

Now let us consider the solution of the auxiliary problem (\ref{e.uk1N}). By
(\ref{e.optc22}), we get $u_{k+1}=0$ if and only if there exists 
$p\in N_C(0)\cap(v_k-s_{k+1} \nabla f(x_{k+1})-s_{k+1} \partial\psi(0))$. From (\ref{e.subdl1}), this is possible if and only if 
\[
\min \left\{ \sum_{p_i<0} p_i \underline{x}+\sum_{p_i>0}p_i \overline{x}~\Big |~ p=v_k-s_{k+1} (\nabla f(x_{k+1})+\mu_f x_{k+1})-s_{k+1} g, 
~\|g\|_\infty \leq 1 \right\} \leq 0.
\]
The solution of this problem is 
$\wh p=v_k-s_{k+1} (\nabla f(x_{k+1})+\mu_f x_{k+1})-s_{k+1}{\bf 1}$. Thus the
minimum of this problem is given by (\ref{e.ind}). This suggests $u_{k+1}=0$ if
and only if $\kappa(\wh p)\leq 0$. 

The definition of $N_C(u_{k+1})$ and (\ref{e.optc22}) imply 
\[
(1+\mu_f s_{k+1}) u_{k+1}^j-v_k^j+s_{k+1} (\nabla f(x_{k+1})+\mu_f x_{k+1})+s_{k+1}(\partial \|u_{k+1}\|_1)^j \left\{
\begin{array}{ll}
\geq 0 &~~ \mathrm{if}~ u_{k+1}^j = \underline{x}^j,\\
\leq 0 &~~ \mathrm{if}~ u_{k+1}^j = \overline{x}^j,\\
= 0    &~~ \mathrm{if}~ \underline{x}^j < u_{k+1}^j < \overline{x}^j,
\end{array}
\right.
\]
for $j = 1, \dots, n$. Equivalently for $u_{k+1} \neq 0$, we get
\lbeq{e.optcl12}
(1+\mu_f s_{k+1})u_{k+1}^j-v_k^j+s_{k+1} (\nabla f(x_{k+1})+\mu_f x_{k+1})+s_{k+1}\mathrm{sign} (u_{k+1}^j) \left\{
\begin{array}{ll}
\geq 0 &~~ \mathrm{if}~ u_{k+1}^j = \underline{x}^j,\\
\leq 0 &~~ \mathrm{if}~ u_{k+1}^j = \overline{x}^j,\\
= 0    &~~ \mathrm{if}~ \underline{x}^j < u_{k+1}^j < \overline{x}^j.
\end{array}
\right.
\eeq
If $u_{k+1}^j = \underline{x}^j$, we have $c_4^j \geq 0$. Substituting
$u_{k+1}^j = \overline{x}^j$ in (\ref{e.optcl12}) implies $c_4^j \leq 0$.
If $\underline{x}^j < u_{k+1}^j < \overline{x}^j$, there are three possibilities: (i) $u_{k+1}^j > 0$; (ii) $u_{k+1}^j < 0$; (iii) $u_{k+1}^j=0$.
In Case (i), $\mathrm{sign} (u_{k+1}^j) = 1$ and (\ref{e.optcl12})
imply $u_{k+1}^j =c_5^j>0$. In Case (ii), $\mathrm{sign} (\wh z^i) = -1$ and (\ref{e.optcl12}) lead to $u_{k+1}^j =c_6^j<0$. In Case (c), we get 
$u_{k+1}^j = 0$.\qed
\end{proof}

A particular case of box constraints is the nonnegativity constraints 
($x\geq 0$) appearing in many applications because $x$ describes some physical
quantities,  see, e.g., \cite{EssLX,KauN2}. Propositions \ref{p.bcopt}
and \ref{p.soll1} can be simplified for nonnegativity constraints.

\vspace{-1mm}
\section{Numerical experiments}
In this section we report some numerical results to compare the performance of  ASGA-1, ASGA-2, ASGA-3, and ASGA-4 with some state-of-the-art solvers. More precisely, we compare them with NSDSG (nonsummable
diminishing subgradient algorithm \cite{BoyXM}), PGA (proximal gradient algorithm \cite{ParB}), FISTA (Beck and Teboulle's fast proximal gradient algorithm \cite{BecT2}), NESCO (Nesterov's composite gradient algorithm
\cite{NesC}), NESUN (Nesterov's universal gradient algorithm \cite{NesU}). 

The codes of all algorithms are written in MATLAB, where the codes of
ASGA-1, ASGA-2, ASGA-3, and ASGA-4 are available at
\begin{center}
\url{http://homepage.univie.ac.at/masoud.ahookhosh/}.
\end{center}
For $\ell_1$ and elastic net minimization (Sections \ref{s.l1} and
\ref{s.elas}), ASGA-2 and ASGA-4 use $\gamma_1=4$ and $\gamma_2=0.9$, while
for support vector machine (Section \ref{s.svm}) ASGA-2 uses $\gamma_1=4$ and $\gamma_2=0.9$ and ASGA-4 uses $\gamma_1=4$ and $\gamma_2=0.6$. The other
considered algorithms use the parameters proposed in the associated literature.
In our implementation, NSDSG uses the step-sizes $\alpha_k:=\alpha_0/\sqrt{k}$, 
where we set $\alpha_0=10^{-1}$ in Sections \ref{s.l1} and \ref{s.elas} and 
$\alpha_0=5\times 10^{-11}$ in Section \ref{s.svm}. All numerical experiments
are executed on a PC Intel Core i7-3770 CPU 3.40GHz 8 GB RAM.

\vspace{-2mm}
\subsection{{\bf $\ell_1$ minimization}}\label{s.l1}
We consider solving the underdetermined system 
\begin{equation}\label{e.line}
A x = y,
\end{equation}
where $A \in \mathbb{R}^{m \times n}$ ($m \leq n$) and $y \in \mathbb{R}^m$. 
Underdetermined system of linear equations is frequently appeared in many applications of linear inverse problem such as those in the fields signal and image processing, geophysics, economics, machine learning, and statistics. The objective is to recover $x$ from the observed vector $y$ and matrix $A$ by some optimization models. Due to the ill-conditioned feature of the problem, a
regularized version of the problem is minimized, cf. \cite{NeuI}. We here 
consider the $\ell_1$ minimization 
\begin{equation}\label{e.l1}
\min_{x\in\mathbb{R}^n} ~~ \D \frac{1}{2} \|y-Ax\|_2^2 + \lambda \|x\|_1,
\end{equation}
where $\lambda > 0$ is a regularization parameter. This is a nonsmooth convex
problem of the form (\ref{e.gfun}) with $f(x)= \frac{1}{2} \|y-Ax\|_2^2$ and 
$\psi(x)=\lambda \|x\|_1$. It is straightforward to see that $f$ is Lipschitz
continuous with $\nu=1$ and $L_\nu=\|A\|_2^2$ implying that ASGA-1 and ASGA-3 
can be applied to this problem.

The problem is generated by
\begin{equation}\label{e.data}
\mathtt{[A,z,x] = i\_laplace(n),~~~ y= z + 0.1*rand,}
\end{equation}
where $n=5000$ is the problem dimension and $\mathtt{i\_laplace.m}$ is an ill-posed problem generator using the inverse Laplace transformation from Regularization Tools package (cf. \cite{Han}), which is available at
\begin{center}
\url{http://www.imm.dtu.dk/~pcha/Regutools/}.
\end{center}
We here run NESUN, ASGA-1, ASGA-2, ASGA-3, and ASGA-4 to solve this $\ell_1$ 
minimization problem. The algorithms are stopped after 30 seconds of the
running time. The results are summarized Table \ref{t.l1}, where $f_b$ and
$f_N$ denote the best function value and the number function evaluations, respectively.

From the results of Table \ref{t.l1} we will see that NESUN, ASGA-2, and
ASGA-4 are more sensitive to regularization parameters than ASGA-1 and ASGA-3;
however, ASGA-1 is much less sensitive than NESUN and ASGA-2. It can also be
seen that NESUN attains the worst results for $\lambda=10$ and $\lambda=1$. For
$\lambda \leq 10^{-1}$, we see that ASGA-2 and ASGA-4 outperform
the others, while NESUN, ASGA-1 and ASGA-3 perform to some extent comparable. During our experiments we spot a disadvantage of NESUN, ASGA-2, and
ASGA-4 which is the sensitivity to the small accuracy parameter $\eps$.
In this case we found out that the associated line search does not terminate
because of the possible round-off error that is a usual problem in Armijo-type
line searches (cf. \cite{AhoG}). For $\lambda=10^{-1}$, we show this in
Subfigures (a) and (b) of Figure \ref{f.l1} with $\eps=10^{-1}$ and 
$\eps=10^{-4}$, respectively. Therefore, it would be much more reliable to
apply ASGA-1 and ASGA-3 for the accuracy parameter smaller than $\eps=10^{-2}$ 
if $\nu$ and $L_\nu$ are available.
 
\begin{table}[htbp]
\caption
{Best function values $f_b$ and the number of function evaluations $N_f$ for
NESUN, ASGA-1, ASGA-2, ASGA-3, and ASGA-4 for solving the $\ell_1$ minimization problem (\ref{e.l1}) with several regularization parameters}
\label{t.l1}
\begin{center}\footnotesize
\renewcommand{\arraystretch}{1.2}
\begin{tabular}{lllllllllllll}\hline
\multicolumn{1}{l}{Reg.par.} \hspace{5mm}
&\multicolumn{1}{l}{{\bf NESUN}} &\multicolumn{1}{l}{} \hspace{5mm} 
&\multicolumn{1}{l}{{\bf ASGA-1}} &\multicolumn{1}{l}{} \hspace{5mm} 
&\multicolumn{1}{l}{{\bf ASGA-2}} &\multicolumn{1}{l}{} \hspace{5mm}  
&\multicolumn{1}{l}{{\bf ASGA-3}} &\multicolumn{1}{l}{} \hspace{5mm}
&\multicolumn{1}{l}{{\bf ASGA-4}} &\multicolumn{1}{l}{} \\ 
\cmidrule(lr){2-3} \cmidrule(lr){4-5} \cmidrule(lr){6-7}
\cmidrule(lr){8-9} \cmidrule(lr){10-11} 
\multicolumn{1}{l}{} 
&\multicolumn{1}{l}{$f_b$} &\multicolumn{1}{l}{$f_N$}
&\multicolumn{1}{l}{$f_b$} &\multicolumn{1}{l}{$f_N$}
&\multicolumn{1}{l}{$f_b$} &\multicolumn{1}{l}{$f_N$}
&\multicolumn{1}{l}{$f_b$} &\multicolumn{1}{l}{$f_N$}
&\multicolumn{1}{l}{$f_b$} &\multicolumn{1}{l}{$f_N$}\\
\hline
$\lambda = 10$ & 3134.81 & 622 & 357.39 & 612 & 363.34 & 611 
& 356.77 & 605 & 3051.31 & 605\\
$\lambda = 1$ & 284.88 & 618 & 60.94 & 623 & 68.46 & 617 
& 60.86 & 605 & 59.76 & 608\\
$\lambda = 10^{-1}$ & 7.89 & 621 & 7.73 & 627 & 7.59 & 603  
& 7.80 & 606 & 6.16 & 601\\
$\lambda = 10^{-2}$ & 1.78 & 656 & 2.64 & 588 & 0.98 & 597 
& 2.57 & 595 & 0.98 & 588\\
$\lambda = 10^{-3}$ & 1.80 & 619 & 1.42 & 609 & 0.92 & 600 
& 1.46 & 587 & 0.97 & 616\\
$\lambda = 10^{-4}$ & 0.21 & 635 & 0.20 & 614 & 0.20 & 639 
& 0.20 & 632 & 0.20 & 613\\
$\lambda = 10^{-5}$ & 0.04 & 641 & 0.03 & 606 & 0.02 & 610 
& 0.03 & 615 & 0.02 & 616\\
\hline
\end{tabular}
\end{center}
\end{table}

\begin{figure}[htbp]
\centering
\subfloat[][$\lambda=10^{-1},~~ \eps=10^{-1}$]
{\includegraphics[width=7.7cm]{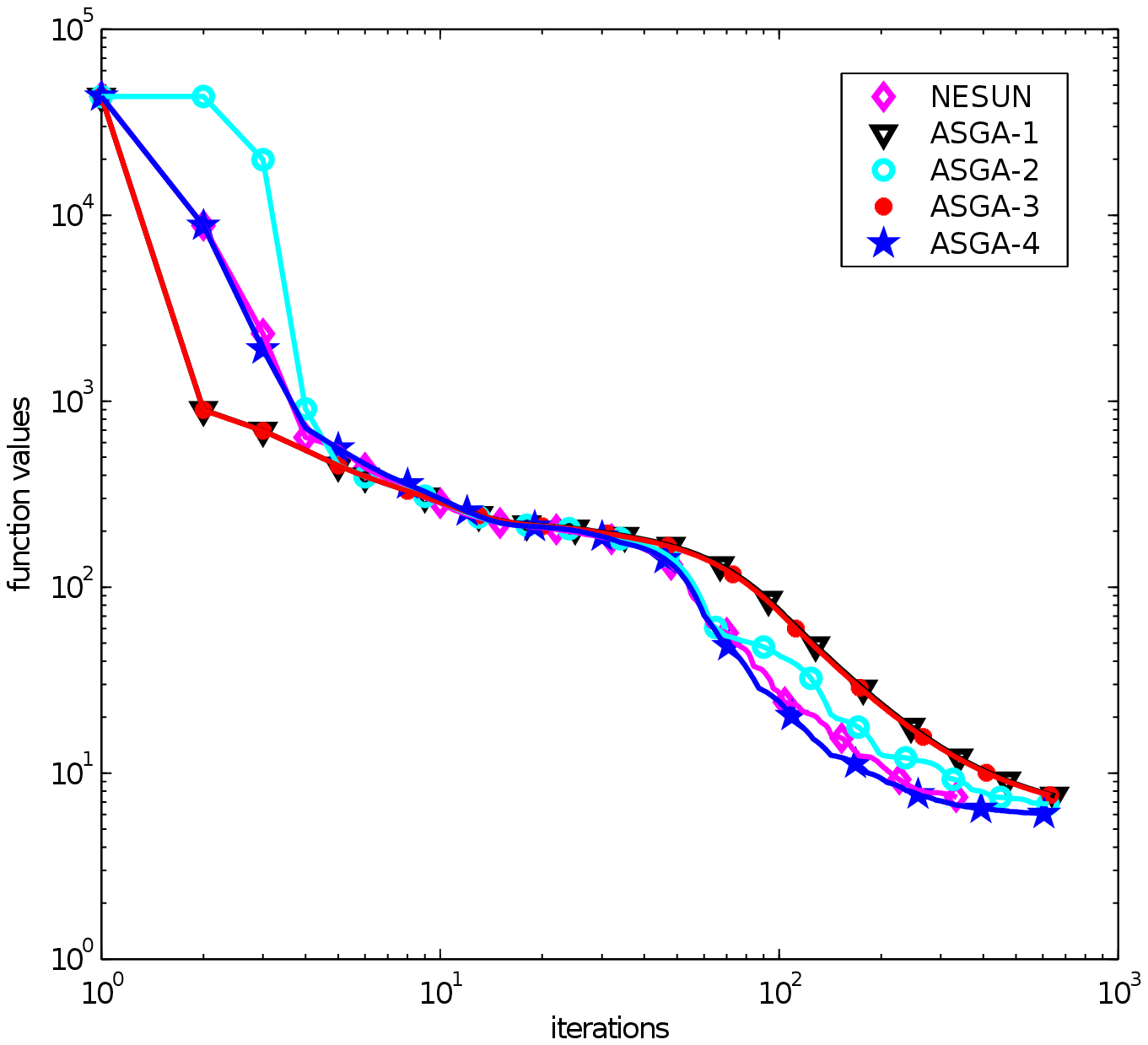}}%
\qquad
\subfloat[][$\lambda=10^{-1},~~ \eps=10^{-4}$]
{\includegraphics[width=7.7cm]{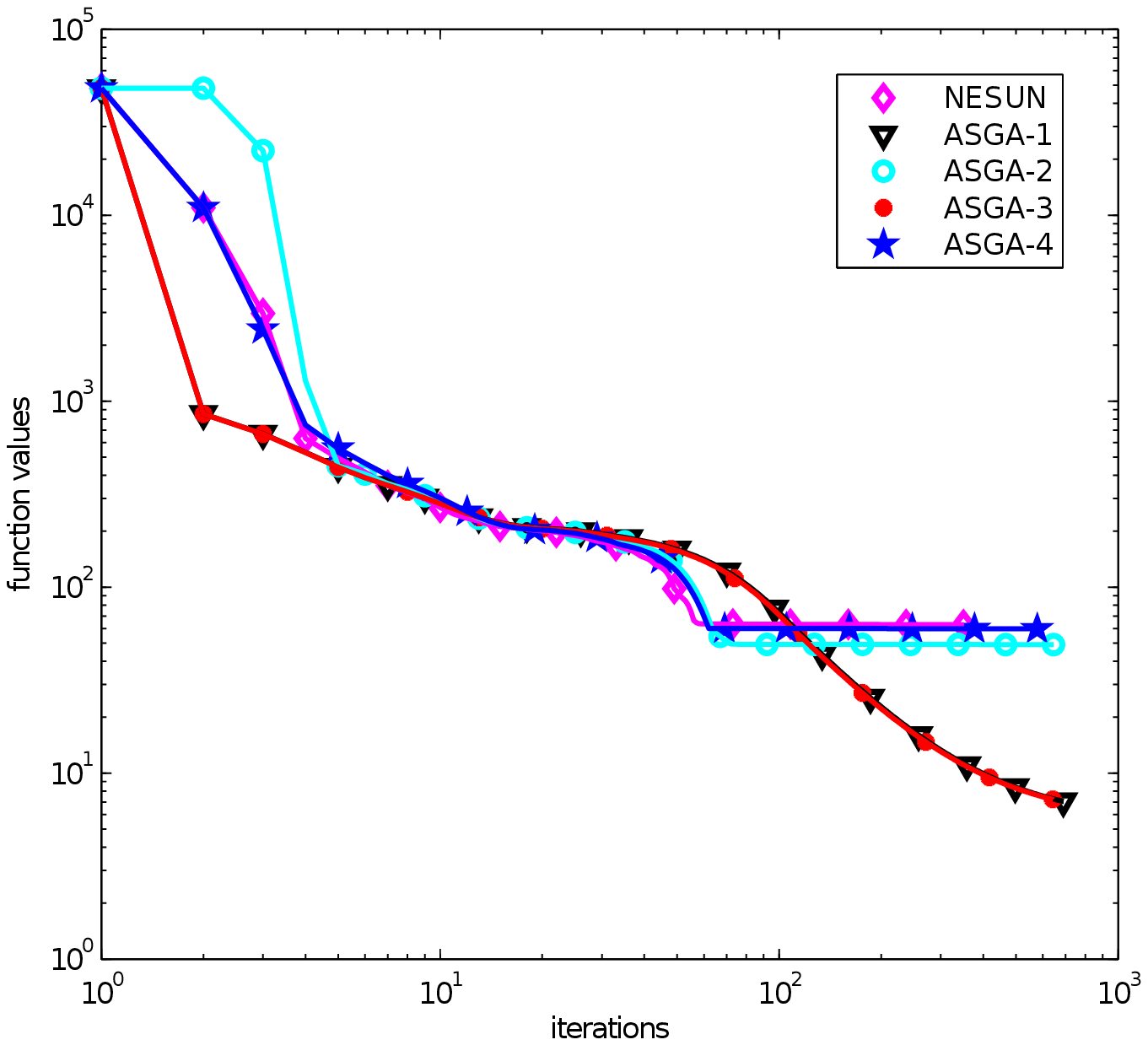}}

\caption{A comparison among NESUN, ASGA-1, ASGA-2, ASGA-3, and ASGA-4 for solving the $\ell_1$ minimization problem (\ref{e.l1}). For $\lambda=10^{-1}$, 
Subfigures (a) and (b) display the results for $\eps=10^{-1}$ and
$\eps=10^{-4}$, respectively. The algorithms stopped after 30 seconds.}
\label{f.l1}
\end{figure}

\vspace{-10mm}
\subsection{{\bf Elastic net minimization}}\label{s.elas}
Let us consider the underdetermined system (\ref{e.line}), where
the data is generated by (\ref{e.data}). Since this problem is 
ill-conditioned, we apply a regularized least-squares with the elastic net regularizer, i.e.,
\lbeq{e.elas1}
\min_{x\in\mathbb{R}^n}~~ \D \frac{1}{2} \|y-Ax\|_2^2 + \frac{1}{2} \lambda_1 \|x\|_2^2 + \lambda_2 \|x\|_1
\eeq 
or
\lbeq{e.elas2}
\bary{ll}
\min &~~ \D \frac{1}{2} \|y-Ax\|_2^2 + \frac{1}{2} \lambda_1 \|x\|_2^2 + \lambda_2 \|x\|_1\\
\st  &~~ x \in \x=[\underline{x},\overline{x}],
\eary
\eeq 
where $\lambda_1, \lambda_2 > 0$ are regularization parameters. This problem
is nonsmooth and strongly convex. By setting 
$f(x)=\frac{1}{2}\|y-Ax\|_2^2+\frac{1}{2}\lambda_1\|x\|_2^2$ and 
$\psi(x)=\lambda_2 \|x\|_1$, we have that $f$ is $\lambda_1$-strongly convex
and has Lipschitz continuous gradients with $\nu=1$ and 
$L_\nu=\|A\|_2^2+\lambda_1$.

We now run NSDSG, PGA, FISTA, NESCO, NESUN, ASGA-1, ASGA-2, ASGA-3, and ASGA-4 for 
solving the elastic net minimization problem (\ref{e.elas1}) and NSDSG, NESCO,
NESUN, ASGA-1, ASGA-2, ASGA-3, and ASGA-4 for solving the box-constrained version
(\ref{e.elas2}). The auxiliary problems of NESCO, NESUN, ASGA-1, ASGA-2, ASGA-3,
and ASGA-4 are solved using the statements of Proposition \ref{p.soll1}. For (\ref{e.elas2}), we set $\mathtt{\x=[-ones(5000,1), ones(5000,1)]}$. We stop the algorithms after 20 seconds of the running time.
The results are summarized in Table \ref{t.l22l1}.

\begin{landscape}
\begin{table}[htbp]
\caption{Numerical results of NSDSG, PGA, FISTA, NESCO, NESUN, ASGA-1, ASGA-2, ASGA-3, and ASGA-4 for the elastic net minimization problems (\ref{e.elas1})
and (\ref{e.elas2}). The first 10 problems stands for (\ref{e.elas1}) and the
remainder for (\ref{e.elas2}). The algorithms were stopped after 20 seconds of
the running time. $P$, $f_b$, and $N_f$ denote the problem number, the best function value, and the number of the function evaluations achieved by the
algorithms, respectively. }
\label{t.l22l1}
\begin{center}\footnotesize
\renewcommand{\arraystretch}{1.85}
\begin{tabular}{lllllllllllllllllllll}
\hline
\multicolumn{1}{l}{$P$}
&\multicolumn{1}{l}{$\lambda_1$} 
&\multicolumn{1}{l}{$\lambda_2$} 
&\multicolumn{1}{l}{NSDSG}  &\multicolumn{1}{l}{}       
&\multicolumn{1}{l}{PGA}    &\multicolumn{1}{l}{} 
&\multicolumn{1}{l}{FISTA}  &\multicolumn{1}{l}{} 
&\multicolumn{1}{l}{NESCO}  &\multicolumn{1}{l}{}  
&\multicolumn{1}{l}{NESUN}  &\multicolumn{1}{l}{} 
&\multicolumn{1}{l}{ASGA-1}  &\multicolumn{1}{l}{} 
&\multicolumn{1}{l}{ASGA-2}  &\multicolumn{1}{l}{} 
&\multicolumn{1}{l}{ASGA-3}  &\multicolumn{1}{l}{}  
&\multicolumn{1}{l}{ASGA-4}  &\multicolumn{1}{l}{} \\ 
\cmidrule(lr){4-5}\cmidrule(lr){6-7} \cmidrule(lr){8-9} \cmidrule(lr){10-11} 
\cmidrule(lr){12-13}\cmidrule(lr){14-15}\cmidrule(lr){16-17}
\cmidrule(lr){18-19}\cmidrule(lr){20-21}

\multicolumn{1}{l}{} &\multicolumn{1}{l}{} & \multicolumn{1}{l}{}        
&\multicolumn{1}{l}{$f_b$} &\multicolumn{1}{l}{$N_f$} 
&\multicolumn{1}{l}{$f_b$} &\multicolumn{1}{l}{$N_f$}
&\multicolumn{1}{l}{$f_b$} &\multicolumn{1}{l}{$N_f$}
&\multicolumn{1}{l}{$f_b$} &\multicolumn{1}{l}{$N_f$}
&\multicolumn{1}{l}{$f_b$} &\multicolumn{1}{l}{$N_f$}
&\multicolumn{1}{l}{$f_b$} &\multicolumn{1}{l}{$N_f$}
&\multicolumn{1}{l}{$f_b$} &\multicolumn{1}{l}{$N_f$}
&\multicolumn{1}{l}{$f_b$} &\multicolumn{1}{l}{$N_f$}
&\multicolumn{1}{l}{$f_b$} &\multicolumn{1}{l}{$N_f$}\\ 
\hline
1 & $10^{-3}$ & $1$ & 364.24 & 854 & 55.79 & 691 & 53.91 & 434
& 58.17 & 497 & 335.58 & 455 & 57.81 & 428 & 71.08 & 424 
& 57.33 & 434 & 59.81 & 420 \\
2 & $10^{-3}$ & $10^{-1}$ & 192.85 & 719 & 114.39 & 673 & 6.46 & 452
& 18.38 & 555 & 14.34 & 480 & 9.54 & 464 & 8.29 & 459 
& 9.13 & 488 & 7.14 & 455 \\
3 & $10^{-3}$ & $10^{-2}$ & 26.98 & 637 & 21.47 & 655 & 2.31 & 472
& 18.60 & 515 & 4.39 & 444 & 4.03 & 440 & 1.13 & 484 
& 3.47 & 478 & 1.43 & 436 \\
4 & $10^{-3}$ & $10^{-3}$ & 6.40 & 696 & 3.40 & 6.38 & 1.99 & 430
& 2.81 & 611 & 2.50 & 415 & 2.14 & 391 & 1.87 & 416 
& 1.99 & 440 & 1.82 & 427 \\
5 & $10^{-3}$ & $10^{-4}$ & 3.96 & 653 & 1.34 & 671 & 0.66 & 541
& 0.98 & 615 & 0.81 & 450 & 0.73 & 428 & 0.69 & 464 
& 0.72 & 443 & 0.70 & 435 \\

6 & $10^{-4}$ & $1$ & 499.59 & 664 & 61.73 & 710 & 59.29 & 430
& 65.17 & 559 & 408.85 & 431 & 62.96 & 459 & 79.11 & 422 
& 62.65 & 467 & 65.22 & 521 \\
7 & $10^{-4}$ & $10^{-1}$ & 194.57 & 654 & 114.54 & 677 & 6.35 & 511
& 19.82 & 555 & 10.21 & 448 & 10.00 & 442 & 9.14 & 433 
& 9.35 & 465 & 6.83 & 446 \\
8 & $10^{-4}$ & $10^{-2}$ & 24.49 & 757 & 20.56 & 674 & 1.67 & 593
& 16.81 & 603 & 4.10 & 456 & 4.27 & 431 & 1.31 & 440 
& 3.57 & 476 & 1.25 & 463 \\
9 & $10^{-4}$ & $10^{-3}$ & 5.26 & 660 & 4.76 & 657 & 1.73 & 446
& 2.33 & 5.85 & 2.04 & 425 & 1.74 & 448 & 1.63 & 433 
& 1.64 & 500 & 1.59 & 442 \\
10 & $10^{-4}$ & $10^{-4}$ & 4.86 & 6.49 & 1.16 & 605 & 0.28 & 420
& 0.67 & 511 & 0.31 & 415 & 0.28 & 412 & 0.28 & 408 
& 0.28 & 405 & 0.28 & 410 \\

11 & $10^{-3}$ & $1$ & 343.03 & 909 & --- & --- & --- & --- & 56.05 & 643 
& 104.16 & 528 & 54.51 & 509 & 65.60 & 557 & 53.98 & 536 & 52.87 & 519 \\
12 & $10^{-3}$ & $10^{-1}$ & 190.73 & 899 & --- & --- & --- & --- & 10.94 & 679 
& 9.44 & 572 & 8.90 & 525 & 8.17 & 516 & 8.56 & 534 & 7.03 & 472 \\
13 & $10^{-3}$ & $10^{-2}$ & 35.49 & 829 & --- & --- & --- & --- & 15.70 & 691 
& 3.38 & 497 & 3.27 & 499 & 1.07 & 513 & 2.82 & 540 & 1.09 & 542 \\
14 & $10^{-3}$ & $10^{-3}$ & 17.63 & 649 & --- & --- & --- & --- & 3.08 & 583
& 2.47 & 509 & 2.01 & 467 & 1.29 & 540 & 1.88 & 506 & 1.42 & 547 \\
15 & $10^{-3}$ & $10^{-4}$ & 10.11 & 903 & --- & --- & --- & --- & 0.97 & 689 
& 0.81 & 541 & 0.73 & 472 & 0.64 & 561 & 0.72 & 475 & 0.64 & 542 \\

16 & $10^{-4}$ & $1$ & 353.73 & 900 & --- & --- & --- & --- & 64.31 & 629 
& 336.19 & 510 & 61.04 & 556 & 70.63 & 532 & 61.26 & 509 & 59.89 & 522 \\
17 & $10^{-4}$ & $10^{-1}$ & 176.44 & 892 & --- & --- & --- & --- & 15.59 & 567 
& 7.67 & 498 & 7.79 & 539 & 6.99 & 540 & 8.37 & 477 & 6.20 & 439 \\
18 & $10^{-4}$ & $10^{-2}$ & 23.37 & 821 & --- & --- & --- & --- & 15.58 & 695 
& 2.42 & 566 & 3.46 & 491 & 1.06 & 521 & 3.02 & 529 & 1.25 & 463 \\
19 & $10^{-4}$ & $10^{-3}$ & 8.20 & 906 & --- & --- & --- & --- & 2.32 & 603 
& 1.92 & 525 & 1.66 & 481 & 1.18 & 548 & 1.57 & 530 & 1.42 & 492 \\
20 & $10^{-4}$ & $10^{-4}$ & 12.70 & 858 & --- & --- & --- & --- & 0.49 & 663 
& 0.29 & 549 & 0.28 & 504 & 0.27 & 564 & 0.28 & 517 & 0.27 & 539 \\
\hline
\end{tabular}
\end{center}
\end{table}
\end{landscape}

\begin{figure}[htbp]
\centering
\subfloat[][$\lambda_1=10^{-3},~~ \lambda_2=10^{-2}$]
{\includegraphics[width=7.7cm]{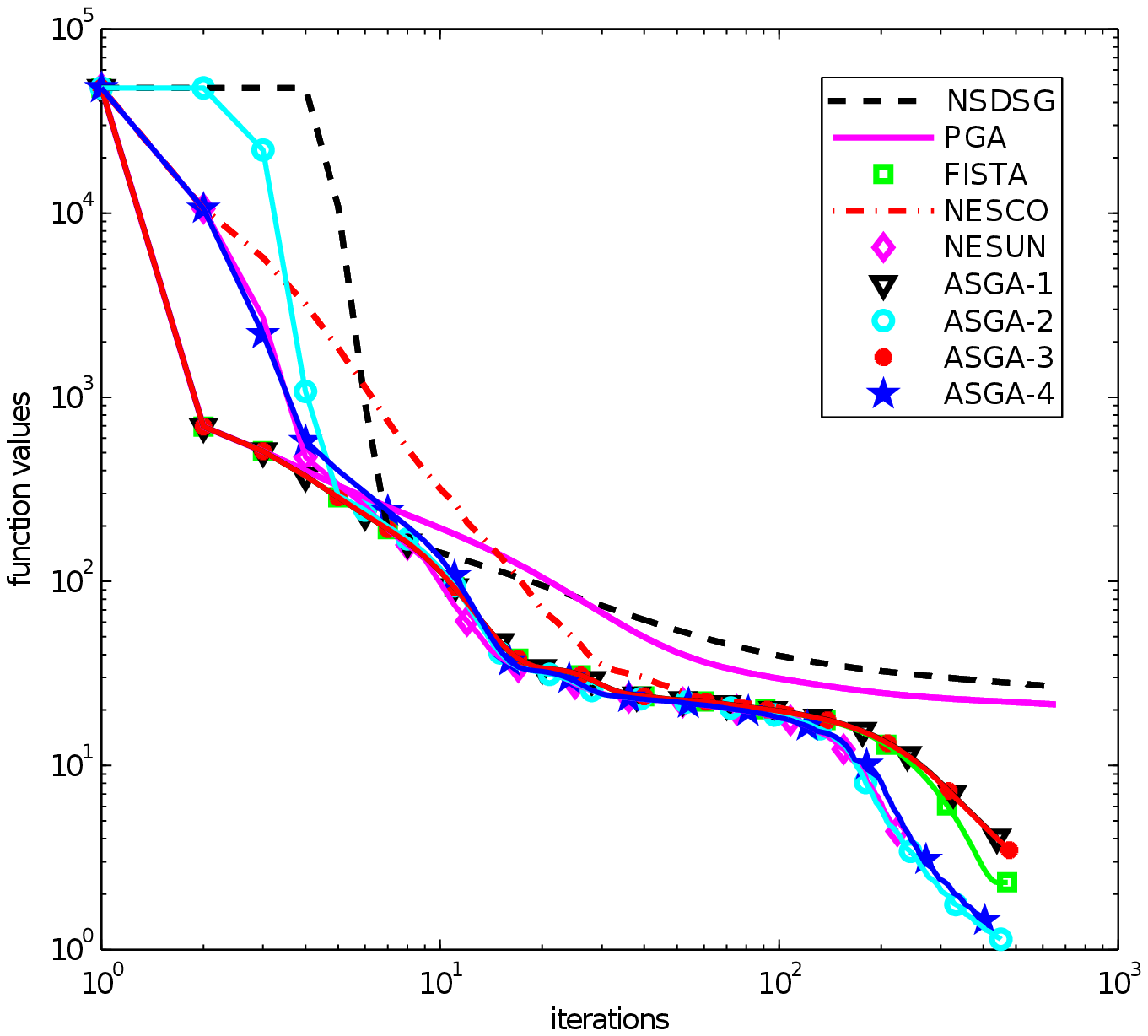}}%
\qquad
\subfloat[][$\lambda_1=10^{-3},~~ \lambda_2=10^{-3}$]
{\includegraphics[width=7.7cm]{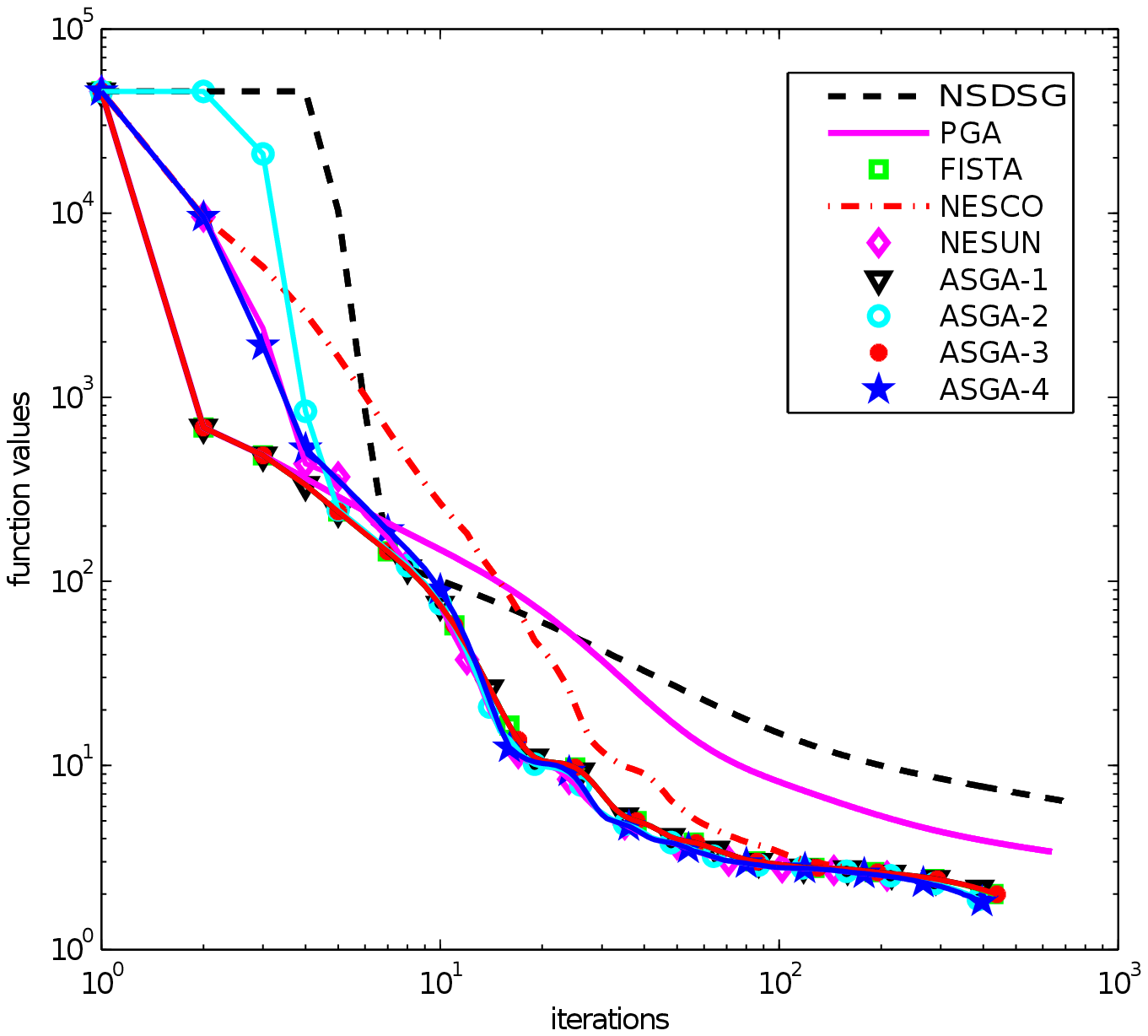}}

\caption{A comparison among NSDSG, PGA, FISTA, NESCO, NESUN, ASGA-1, ASGA-2,
ASGA-3, and ASGA-4 for solving elastic net minimization problems (\ref{e.elas1}): Subfigures (a) and (b) display comparisons of function
values versus iterations for $\lambda_1=10^{-3}, \lambda_2=10^{-2}$ and $\lambda_1=10^{-3}, \lambda_2=10^{-3}$, respectively. The algorithms stopped
after 20 seconds.}
\label{f.ll2l1}
\end{figure}

The results of Table \ref{t.l22l1} shows that the optimal methods FISTA, NESCO,
NESUN, ASGA-1, ASGA-2, ASGA-3, and ASGA-4 outperforms NSDSG and PGA significantly
as confirmed by their complexity analyses. It can also be seen that in many
cases ASGA-2 and ASGA-4 performs better than NSDSG, PGA, FISTA, NESCO, NESUN,
ASGA-1, and ASGA-2; however, in several cases they are comparable with FISTA,
where FISTA is not generally applicable for constrained version
(\ref{e.elas2}). In addition, it is observable that ASGA-1 and ASGA-3 stay
reasonable for a wide range of regularization parameters in contrast to ASGA-2 and ASGA-4. We therefore draw your attention to the Subfigures (a) and (b) of
Figure \ref{f.ll2l1} which give the function values versus iterations for (\ref{e.elas1}) with two levels of regularization parameters 
$\lambda_1=10^{-3}, \lambda_2=10^{-2}$ and 
$\lambda_1=10^{-3}, \lambda_2=10^{-3}$.

\subsection{{\bf Support vector machine}}\label{s.svm}
Let us consider learning with support vector machines (SVM) leading to a convex
optimization problem with large data sets. In particular, we consider a  binary
classification, where the set of training data $(x_1,y_1),\dots,(x_m,y_m)$
with $x_i\in \mathbb{R}^n$ and $y_i\in \{-1, 1\}$, for $i=1, \dots, m$, are
given. The aim is to find a classification rule from the training data, so that when given a new input $x$, we can assign a class $y\in\{-1, 1\}$ to it.
As SVM uses a classification rule that decides the class of $x$ based on the
sign of $\langle x,w\rangle+w_0$, we need to choose the vector $w$ and the
scalar $w_0$. These may be determined by solving the penalized problem
\lbeq{e.svm}
\begin{array}{ll}
\min   &~~ \D\sum_{i=1}^m \left[ 1-y_i(\langle x_i, w\rangle +w_0)\right]_+ +\lambda \phi(w)\\
\mathrm{s.t.} &~~ w \in \mathbb{R}^n,~ w_0 \in \mathbb{R},
\end{array}
\eeq
where $[z]_+ = \max \{z, 0\}$, and $\phi$ can be $\|\cdot\|_1$ (SVML1R),
$\|\cdot\|_2^2$ (SVML22R), and $\frac{1}{2}\|\cdot\|_2^2+\|\cdot\|_1$
(SVML22L1R) (see, e.g., \cite{ShaS,ZhuRHT} and references therein). For 
$\langle x,w \rangle=w^Tx$, let us define 
\[
X:=\left(
\begin{array}{c}
y_1x_1^T\\
\vdots\\
y_m x_m^T
\end{array} 
\right)\in \mathbb{R}^{m\times n},~~~
A:=(X, y)\in \mathbb{R}^{m\times (n+1)},~~~ 
\wt w:=\left( 
\begin{array}{l}
w\\
w_0
\end{array} \right)\in \mathbb{R}^{n+1}.
\]

The problem (\ref{e.svm}) can be rewritten in the form
\lbeq{e.svm1}
\begin{array}{ll}
\min &~~\D \langle \mathbf{1}, \left[ \mathbf{1}-A \wt w \right]_+ \rangle+ \lambda\phi(w)\\
\mathrm{s.t.} &~~ \wt w \in \mathbb{R}^{n+1},
\end{array}
\eeq
where $\left[ \mathbf{1}-A \wt w \right]_+ = \sup \{\mathbf{1}-A \wt w,0\}$
and $\mathbf{1}\in \mathbb{R}^{m}$ is the vector of all ones. Typically $A$
is a dense matrix constructed by data points $x_i$ and $y_i$ for $i=1, \dots, m$. By setting 
$f(\wt w)=\langle \mathbf{1}, \left[ \mathbf{1}-A \wt w \right]_+\rangle$ and
$\psi(x)=\lambda\phi(w)$, it is clear that (\ref{e.svm1}) is of the form
(\ref{e.gfun}), where $f$ is nonsmooth and its corresponding subgradient at 
$\wt w$ is given by
\[
\nabla f(\wt w) = -A^T \delta, 
\]
with
\[
\forall i=1, \dots, m, ~~~ \delta_i := \left\{
\begin{array}{ll}
1 & ~~ \mathrm{if}~ A_{i:} \wt w<1,\\
0 & ~~ \mathrm{if}~ A_{i:} \wt w\geq 1,
\end{array}
\right.
\]
For all $w_1, w_2 \in \mathbb{R}^n$, we have
\[
\|\nabla f(w_1)-\nabla f(w_2)\|_2 = \|A^T (\delta_1-\delta_2)\|_2
\leq \|A^T\|_2 \|\delta_1-\delta_2\|_2 \leq \sqrt m \|A^T\|_2:= L_0,
\]
where $A_{i:}$ denotes the $i$th row of $A$, for $i=1,\dots,m$. Therefore,
$f$ satisfies (\ref{e.holder}) with $\nu=0$ and $L_\nu=L_0$.

Let us consider the problems SVML1R, SVML22R, and SVML22L1R for the leukemia
data given by {\sc Golub} et al. in \cite{Golub}, available at the website
\cite{Golub1}. This dataset comes from a study of gene expression in two types
of acute leukemias (acute myeloid leukemia (AML) and acute lymphoblastic
leukemia (ALL)) and it consists of 38 training data points and 34 test data points. We apply SVML1R, SVML22R, and SVML22L1R to the training data points ($q=38$ and $n=7129$) with six levels of regularization parameters for each of
SVML1R, SVML22R, and SVML22L1R. Since for SVML1R and SVML22L1R both $f$ and
$\psi$ are nonsmooth functions, the algorithms PGA, FISTA, and NESCO cannot be
applied to theses problems. Therefore, we only consider NSDSG, NESUN,
ASGA-1, ASGA-2, ASGA-3, and ASGA-4 for solving these 3 problems with six levels
of regularization parameters. In our implementation, the algorithms are stopped
after 3 seconds of the running time. The associated results are given in Table \ref{t.svm} and Figure \ref{f.svm}.

In spite of the fact that all the considered algorithms attain the complexity
$\mathcal{O}(\eps^{-2})$ for the problem (\ref{e.svm1}), the results of 
Table \ref{t.svm} show that NESUN, ASGA-1, ASGA-2, ASGA-3, and ASGA-4 outperform NSDSG significantly, for all three
problems (SVML1R, SVML22R, and SVML22L1R). For cases $\lambda \in \{10, 1\}$,
NESUN and ASGA-4 attain the better results than the others. For 
$\lambda \in \{10^{-1}, 10^{-2}, 10^{-3},10^{-4}\}$ and for all three problems,
NESUN, ASGA-2, and ASGA-4 perform comparable but better than ASGA-1 and ASGA-3. 
However, ASGA-2 outperforms NESUN and ASGA-4 in the later case. We display
the function values versus iterations of the considered algorithms in 
Subfigures (a) and (b) of Figure \ref{f.svm} for SVML22L1R with $\lambda=1$
and $\lambda=10^{-1}$, respectively. In Subfigure (a), NESUN and ASGA-4
outperform the others, while in Subfigure (b) ASGA-2 possesses the best
result.

\begin{figure}[!htbp]
\centering
\subfloat[][SVML22L1R,~~ $\lambda=1$]
{\includegraphics[width=7.7cm]{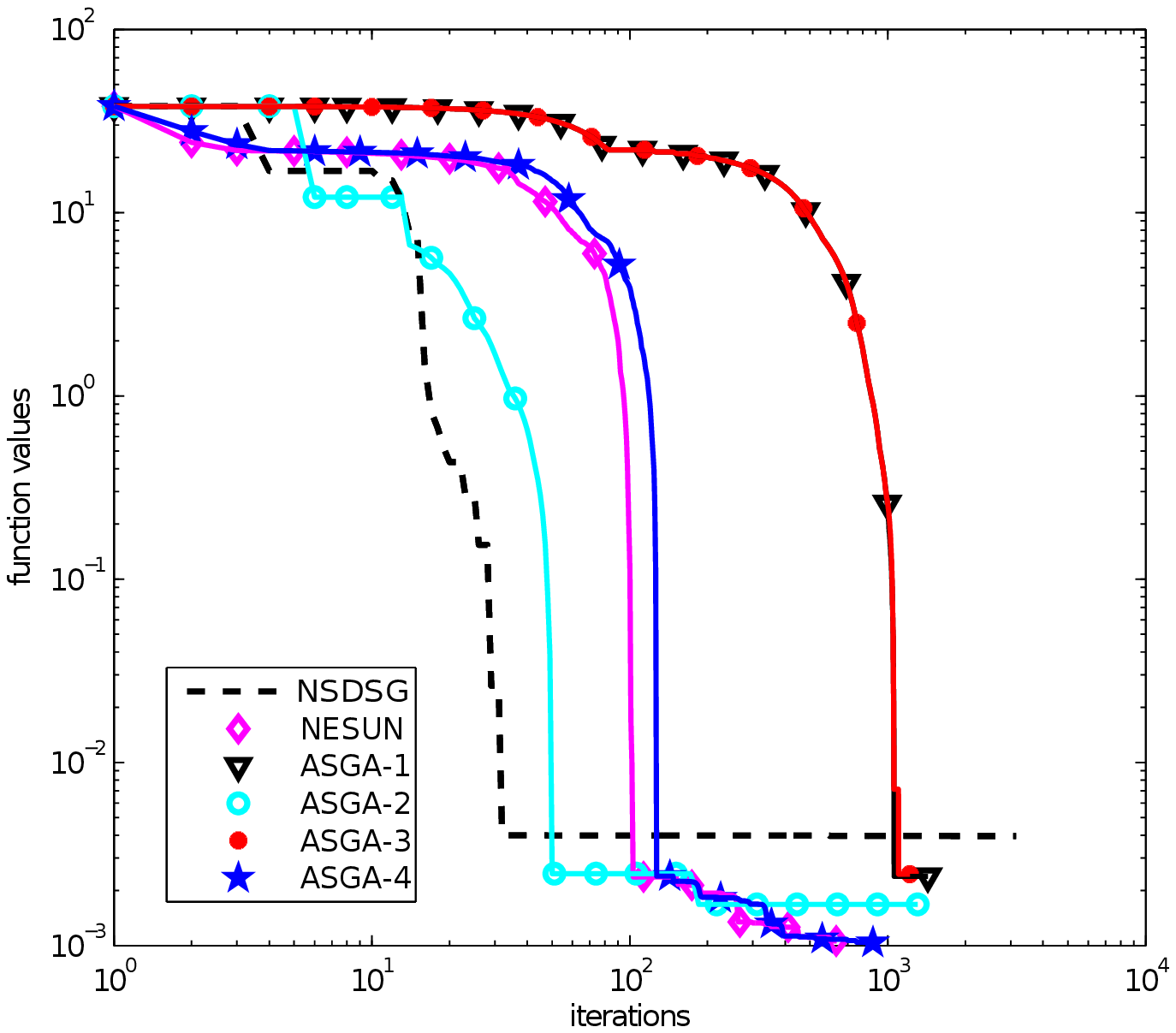}}%
\qquad
\subfloat[][SVML22L1R,~~ $\lambda=10^{-1}$]
{\includegraphics[width=7.7cm]{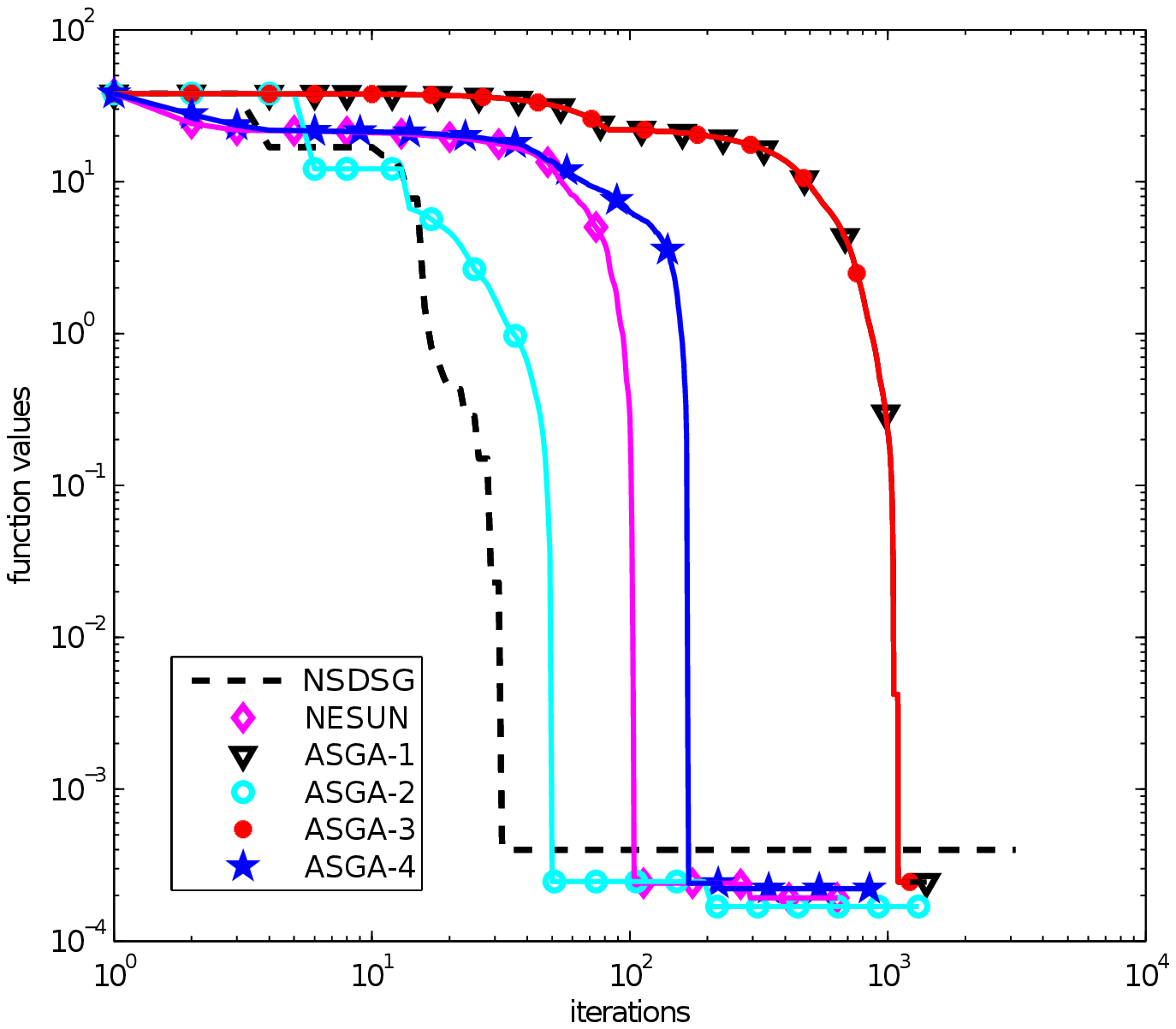}}%

\caption{A comparison among NSDSG, NESUN, ASGA-1, ASGA-2, ASGA-3, and ASGA-4
for a binary classification with linear support vector machines (SVML22L1R)
with $\lambda=1$ and $\lambda=10^{-1}$. The algorithms were stopped after 3
seconds.}
\label{f.svm}
\end{figure}

\begin{landscape}
\begin{table}[htbp]
\caption{Numerical results of NSDSG, NESUN, ASGA-1, ASGA-2, ASGA-3, and ASGA-4 for the binary classification with linear support vector machines (\ref{e.svm1}). The algorithms were stopped after 3 seconds of the running time. $f_b$ and $N_f$ denote the best function value and the number of function
evaluations. }
\label{t.svm}
\begin{center}\footnotesize
\renewcommand{\arraystretch}{2}
\begin{tabular}{llllllllllllll}
\hline
\multicolumn{1}{l}{Prob. name} \hspace{5mm} 
&\multicolumn{1}{l}{Reg. par.} \hspace{5mm}
&\multicolumn{1}{l}{NSDSG}    &\multicolumn{1}{l}{} \hspace{2mm} 
&\multicolumn{1}{l}{NESUN}    &\multicolumn{1}{l}{} \hspace{2mm}       
&\multicolumn{1}{l}{ASGA-1}    &\multicolumn{1}{l}{} \hspace{2mm}  
&\multicolumn{1}{l}{ASGA-2}    &\multicolumn{1}{l}{} \hspace{2mm}   
&\multicolumn{1}{l}{ASGA-3}    &\multicolumn{1}{l}{} \hspace{2mm}   
&\multicolumn{1}{l}{ASGA-4}    &\multicolumn{1}{l}{}\\ 
\cmidrule(lr){3-4}\cmidrule(lr){5-6} \cmidrule(lr){7-8} 
\cmidrule(lr){9-10}\cmidrule(lr){11-12}\cmidrule(lr){13-14} 
\multicolumn{1}{l}{}   & \multicolumn{1}{l}{}        
&\multicolumn{1}{l}{$f_b$} &\multicolumn{1}{l}{$N_f$} \hspace{2mm}
&\multicolumn{1}{l}{$f_b$} &\multicolumn{1}{l}{$N_f$} \hspace{2mm}
&\multicolumn{1}{l}{$f_b$} &\multicolumn{1}{l}{$N_f$} \hspace{2mm}
&\multicolumn{1}{l}{$f_b$} &\multicolumn{1}{l}{$N_f$} \hspace{2mm}
&\multicolumn{1}{l}{$f_b$} &\multicolumn{1}{l}{$N_f$} \hspace{2mm}
&\multicolumn{1}{l}{$f_b$} &\multicolumn{1}{l}{$N_f$}\\ 
\hline
SVML1R    & $\lambda=10$ & 3.63 $\times 10^{-2}$ & 3641 & 9.97$\times 10^{-3}$ & 1482 & 2.43$\times 10^{-2}$ & 1682 & 1.51$\times 10^{-2}$ & 1524 & 
2.37$\times 10^{-2}$ & 1247 & 1.21$\times 10^{-2}$ & 1239 \\ 
SVML1R    & $\lambda=1$ & 3.95$\times 10^{-3}$ & 3389 & 1.07$\times 10^{-3}$ 
& 1347 & 2.38$\times 10^{-3}$ & 1547 & 1.68$\times 10^{-3}$ & 1483 & 
2.45$\times 10^{-3}$ & 1177 & 1.13$\times 10^{-3}$ & 1179 \\
SVML1R    & $\lambda=10^{-1}$ & 3.99$\times 10^{-4}$ & 3401 & 
1.92$\times 10^{-4}$ & 1472 & 2.45$\times 10^{-4}$ & 1498 & 
1.68$\times 10^{-4}$ & 1439 & 2.45$\times 10^{-4}$ & 1223 & 
2.21$\times 10^{-4}$ & 1302 \\
SVML1R    & $\lambda=10^{-2}$ & 3.99$\times 10^{-5}$ & 3408 & 
1.66$\times 10^{-5}$ & 1357 & 2.38$\times 10^{-5}$ & 1543 & 
1.68$\times 10^{-5}$ & 1385 & 2.38$\times 10^{-5}$ & 1295 & 
1.81$\times 10^{-5}$ & 1264 \\
SVML1R    & $\lambda=10^{-3}$ & 3.99$\times 10^{-6}$ & 3326 & 
1.94$\times 10^{-6}$ & 1353 & 2.45$\times 10^{-6}$ & 1530 & 
1.67$\times 10^{-6}$ & 1404 & 2.45$\times 10^{-6}$ & 1304 & 
1.77$\times 10^{-6}$ & 1254 \\
SVML1R    & $\lambda=10^{-4}$ & 3.99$\times 10^{-7}$ & 3362 & 
1.84$\times 10^{-7}$ & 1396 & 2.45$\times 10^{-7}$ & 1542 & 
1.66$\times 10^{-7}$ & 1475 & 2.45$\times 10^{-7}$ & 1263 & 
1.81$\times 10^{-7}$ & 1269 \\

SVML22R    & $\lambda=10$ & 7.96$\times 10^{-8}$ & 3442 & 2.75$\times 10^{-8}$
& 1454 & 2.91$\times 10^{-8}$ & 1598 & 3.01$\times 10^{-8}$ & 1490 & 
2.91$\times 10^{-8}$ & 1236 & 2.76$\times 10^{-8}$ & 1276 \\ 
SVML22R    & $\lambda=1$ &  7.96$\times 10^{-9}$ & 3445 & 2.75$\times 10^{-9}$
& 1454 & 2.91$\times 10^{-9}$ & 1609 & 3.01$\times 10^{-9}$ & 1506 & 
2.91$\times 10^{-9}$ & 1384 & 2.76$\times 10^{-9}$ & 1398 \\
SVML22R    & $\lambda=10^{-1}$ &  7.96$\times 10^{-10}$ & 3438 & 2.75$\times 10^{-10}$
& 1417 & 2.91$\times 10^{-10}$ & 1565 & 3.01$\times 10^{-10}$ & 1479 & 
2.91$\times 10^{-10}$ & 1373 & 2.76$\times 10^{-10}$ & 1386 \\
SVML22R    & $\lambda=10^{-2}$ &  7.96$\times 10^{-11}$ & 3407 & 2.75$\times 10^{-11}$
& 1434 & 2.91$\times 10^{-11}$ & 1512 & 3.01$\times 10^{-11}$ & 1452 & 
2.91$\times 10^{-11}$ & 1307 & 2.76$\times 10^{-11}$ & 1315 \\
SVML22R    & $\lambda=10^{-12}$ &  7.96$\times 10^{-12}$ & 3498 & 2.75$\times 10^{-12}$
& 1443 & 2.91$\times 10^{-12}$ & 1590 & 3.01$\times 10^{-12}$ & 1564 & 
2.91$\times 10^{-12}$ & 1328 & 2.76$\times 10^{-12}$ & 1257 \\
SVML22R    & $\lambda=10^{-13}$ &  7.96$\times 10^{-13}$ & 3536 & 2.75$\times 10^{-13}$
& 1373 & 2.91$\times 10^{-13}$ & 1521 & 3.01$\times 10^{-13}$ & 1479 & 
2.91$\times 10^{-13}$ & 1315 & 2.76$\times 10^{-13}$ & 1343 \\

SVML22L1R    & $\lambda=10$ & 3.65$\times 10^{-2}$ & 3179 & 
1.01$\times 10^{-2}$ & 1247 & 2.43$\times 10^{-2}$ & 1393
& 1.51$\times 10^{-2}$ & 1192 & 2.37$\times 10^{-2}$ & 1156 & 
1.23$\times 10^{-2}$ & 1166 \\ 
SVML22L1R    & $\lambda=1$ & 3.95$\times 10^{-3}$ & 3145 & 
1.05$\times 10^{-3}$ & 1286 & 2.38$\times 10^{-3}$ & 1431 & 
1.68$\times 10^{-3}$ & 1395 & 2.45$\times 10^{-3}$ & 1219 & 
1.03$\times 10^{-3}$ & 1207 \\
SVML22L1R    & $\lambda=10^{-1}$ & 3.99$\times 10^{-4}$ & 3127 & 
1.92$\times 10^{-4}$ & 1285 & 2.45$\times 10^{-4}$ & 1414 & 
1.68$\times 10^{-4}$ & 1403 & 2.45$\times 10^{-4}$ & 1219 
& 2.21$\times 10^{-4}$ & 1172 \\
SVML22L1R    & $\lambda=10^{-2}$ & 3.99$\times 10^{-5}$ & 3180 & 
1.61$\times 10^{-5}$ & 1273 & 2.38$\times 10^{-5}$ & 1425 & 
1.68$\times 10^{-5}$ & 1364 & 2.38$\times 10^{-5}$ & 1171 & 
1.81$\times 10^{-5}$ & 1111 \\
SVML22L1R    & $\lambda=10^{-3}$ & 3.99$\times 10^{-6}$ & 3156 & 
1.94$\times 10^{-6}$ & 1325 & 2.45$\times 10^{-6}$ & 1472 & 
1.67$\times 10^{-6}$ & 1382 & 2.45$\times 10^{-6}$ & 1144 & 
1.77$\times 10^{-6}$ & 1248 \\
SVML22L1R    & $\lambda=10^{-4}$ & 3.99$\times 10^{-7}$ & 3180 & 
1.84$\times 10^{-7}$ & 1287 & 2.45$\times 10^{-7}$ & 1420 & 
1.66$\times 10^{-7}$ & 1295 & 2.45$\times 10^{-7}$ & 1174 & 
1.82$\times 10^{-7}$ & 1197 \\
\hline
\end{tabular}
\end{center}
\end{table}
\end{landscape}

\section{Final remarks} 
In this paper, we propose several novel (sub)gradient methods for solving
large-scale convex composite minimization. More precisely, we give two 
estimation sequences approximating the objective function with some local and
global information of the objective. For each of the estimation sequences, we
give two iterative schemes attaining the optimal complexities for smooth, nonsmooth, weakly smooth, and smooth strongly convex problems. These schemes
are optimal up to a logarithmic factors for nonsmooth strongly convex problems,
and for weakly smooth strongly convex problems they attain a much better 
complexity than the complexity for weakly smooth convex problems. For each
estimation sequence, the first scheme needs to know about the level of
smoothness and the H\"older constant, while the second one is parameter-free (except for the strong convexity parameter which we set zero if it is not
available) at the price of applying a backtracking line search. We then
consider solutions of the auxiliary problems appearing in these four
schemes and study the important cases appearing in applications that can be solved efficiently either in a closed form or by a simple iterative scheme. 
Considering some applicationsin the fields of sparse optimization and machine
learning, we report numerical results showing the encouraging behavior of the
proposed schemes.\bigskip\\
{\bf Acknowledgement.} I would like to thank Arnold Neumaier for his useful comments on this paper.


\end{document}